\documentclass[leqno, 10pt]{amsart}
\usepackage{enumerate,tikz-cd, amscd, mathtools, float, hyperref}

\tikzcdset{scale cd/.style={every label/.append style={scale=#1},
    cells={nodes={scale=#1}}}}
\usepackage{extpfeil}
\usetikzlibrary{patterns}
\hypersetup{colorlinks,linkcolor={blue},citecolor={blue},urlcolor={cyan}} 
\usepackage{cleveref}
	
\newtheorem{theorem}{Theorem}[section]

\newtheorem{proposition}[theorem]{Proposition}
\newtheorem{lemma}[theorem]{Lemma}
\newtheorem{definition}[theorem]{Definition}
\newtheorem{corollary}[theorem]{Corollary}

\newtheorem{example}[theorem]{Example}

\newtheorem{ex}[theorem]{Examples}
\newtheorem{remark}[theorem]{Remark}

\DeclareMathOperator{\Hom}{Hom}

\DeclareMathOperator{\RMod}{mod-}
\DeclareMathOperator{\End}{End}
\DeclareMathOperator{\Ent}{Ent}
\DeclareMathOperator{\Exit}{Exit}

\newcommand\floor[1]{\lfloor #1 \rfloor}
\newcommand{\newterm}{\textsf}

\newcommand{\congto}{\xto{\sim}}

\DeclareMathOperator{\im}{im}

\DeclareMathOperator{\Ext}{Ext}

\DeclareMathOperator{\Sh}{Sh}

\newcommand{\cO}{\mathcal{O}}

\renewcommand{\cong}{\simeq}

\newcommand{\A}{\mathbb{A}}

\newcommand{\C}{\mathbb{C}}
\newcommand{\D}{\mathbb{D}}

\newcommand{\R}{\mathbb{R}}
\newcommand{\G}{\mathcal{G}}
\newcommand{\Z}{\mathbb{Z}}

\newcommand{\F}{\mathcal{F}}
\newcommand{\T}{\mathbb{T}}
\renewcommand{\P}{\mathbb{P}}

\newcommand{\cP}{P}

\newcommand{\xto}{\xrightarrow}

\newcommand{\cF}{\mathcal F }
\newcommand{\bea}{\begin{eqnarray*} }
\newcommand{\eea}{\end{eqnarray*} }
\newcommand{\be}{\begin{equation} }
\newcommand{\ee}{\end{equation} }
\newcommand{\bp}{\begin{proposition}}
\newcommand{\ep}{\end{construction}}
\newcommand{\bpw}{\begin{construction}}
\newcommand{\epw}{\end{proposition}}
\newcommand{\bt}{\begin{maintheo}}
\newcommand{\et}{\end{maintheo}}
\newcommand{\bpf}{\begin{proof}}
\newcommand{\epf}{\end{proof}}
\newcommand{\bl}{\begin{lemma}}
\newcommand{\el}{\end{lemma}}
\newcommand{\bc}{\begin{corollary}}
\newcommand{\ec}{\end{corollary}}
\newcommand{\bd}{\begin{definition}}
\newcommand{\ed}{\end{definition}}
\newcommand{\bee}{\begin{eqnarray} }
\newcommand{\eee}{\end{eqnarray} }
\newcommand{\brem}{\begin{remark}}
\newcommand{\erem}{\end{remark}}
\newcommand{\bex}{\begin{example}}
\newcommand{\eex}{\end{example}}
\newcommand{\bma}{\begin{bmatrix}}
\newcommand{\ema}{\end{bmatrix}}
\newcommand{\bcs}{\begin{cases}}
\newcommand{\ecs}{\end{cases}}
\newcommand{\bcd}{\begin{tikzcd}}
\newcommand{\ecd}{\end{tikzcd}}

\newcommand{\RR}{\mathbb R}
\newcommand{\ZZ}{\mathbb Z}

\newtheorem{maintheo}{Theorem}

\theoremstyle{definition}

\theoremstyle{plain}

\begin{document}
\title[Resolutions via CCC]{Line bundle Resolutions via the Coherent-Constructible Correspondence} 
\author{David Favero and Mykola Sapronov}
\date{\today}
\address{University of Minnesota \\ School of Mathematics}
\email{favero@umn.edu}
\email{sapro003@umn.edu}
\maketitle

\begin{abstract}
We consider a finite collection of line bundles  $\Phi$  introduced by Bondal on a smooth, projective toric variety $X$. For any coherent sheaf $F$ on $X$, we construct minimal resolutions of $F$ by line bundles in $\Phi$, up to twist, with length bounded by the dimension of $X$ and provide explicit formulae for their Betti numbers. For a toric subvariety $Y \subset X$ of codimension $k$, we give a construction of the minimal resolution of $f_{*}\mathcal{O}_{Y}$ of length $k$ by line bundles in $\Phi$ and relate their Betti numbers to the topology of a stratified real torus. Additionally, we recover a (generally non-minimal) cellular resolution of $f_{*}\mathcal{O}_{Y}$ constructed in \cite{HHL}. Aspects of our proof run through the Coherent Constructible Correspondence, a form of homological mirror symmetry for toric varieties. 
\end{abstract}

\tableofcontents

\section{Introduction}
Homological algebra and resolutions are now fundamental aspects of algebraic geometry, commutative algebra, and topology.
Inspired by ideas from homological mirror symmetry, this paper explores how resolutions of coherent sheaves on a smooth projective toric variety can be described using resolutions of constructible sheaves on the mirror real torus.  

The idea of constructing resolutions from a blend of topological and combinatorial methods is not new. Notably, the pioneering work of Bayer, Strumfels, and Peeva demonstrated how free resolutions of monomial ideals come from certain naturally arising regular cell complexes. A consequence of this perspective is that it allows one to express algebraic Betti numbers in terms of topological Betti numbers coming from reduced homology.  
 
Another place one finds connections between algebra, algebraic geometry, topology, and symplectic geometry is mirror symmetry.  A homological manifestation of this phenomenon is captured by the Coherent-Constructible Correspondence \cite{FLTZ}.  It provides a categorical equivalence between complexes of coherent sheaves on toric varieties and complexes of constructible sheaves on the mirror torus.  This begs the question of whether one can use this correspondence to transfer resolutions from one side to the other.

The answer is of course yes, as we will demonstrate.  However, the CCC is not entirely topological as the constructible sheaf categories involve the notion of singular support which requires a smooth structure. To work purely topologically, we consider a larger category which allows more singular support.  This category can be identified with the category of constructible sheaves on a stratified torus due to Bondal \cite{B}. It admits a nice description in terms of modules over an explicit finite dimensional algebra \cite{HPA}. As such, it has very well-understood homological properties and algorithmic constructions for projective resolutions.  Furthermore, the topological nature of these algebras bounds their projective dimension \cite{Rouquier}. One almost immediate consequence is the following:

\begin{theorem}\label{main theorem 1}
Any coherent sheaf on a smooth projective $n$-dimensional toric variety admits a minimal resolution of length at most $n$ by sums of line bundles.
\end{theorem}

A major inspiration for this project is the recent work of Hanlon--Hicks--Lazarev \cite{HHL} who constructed cellular line bundle resolutions of structure sheaves of toric subvarieties.  In fact, their work can be used to recover the statement above albeit by larger resolutions than the ones produced here and hence, as they demonstrate it resolves a conjecture of Berkesch--Erman--Smith \cite[Question 6.5]{BES}.

It turns out that we can fully recover the Hanlon--Hicks--Lazarev resolutions using our framework.  We show they come from a natural cellular refinement of Bondal's stratification. More precisely, they can be obtained by considering minimal injective resolutions over the incidence algebra of this cellular refinement complex.  

Moreover, one need not work cellularly.  Instead, one can directly consider Bondal's original stratification and the associated finite dimensional algebra.  The benefit of this perspective is that it produces smaller, minimal resolutions which, as we show, are unique under certain conditions.  Although the differential is less well-behaved, the terms are explicitly controlled by compactly supported cohomology of the strata.   This is a direct generalization of the cellular complex of Hanlon-Hicks-Lazarev as the compactly supported cohomology of an open $k$-cell is concentrated in degree $k$.  Hence, the following can be interpreted as a minimized version of the Hanlon--Hicks--Lazarev theorem.

\begin{theorem}\label{main theorem 2}
Let $f: X_{\Sigma_1} \to X_{\Sigma_2}$ be a finite toric morphism of smooth projective toric varieties. The sheaf $f_{*}\cO_{X_{\Sigma_1}}$ admits a minimal resolution of length $\dim X_{\Sigma_2} - \dim X_{\Sigma_1}$ by sums of line bundles:
\[
0 \to \bigoplus_{[a] \in \im{\Phi}} \cO(-a)^{\oplus \beta_{k,-a}}
\to ... \to  \bigoplus_{[a] \in \im{\Phi}} \cO(-a)^{\oplus \beta_{0,-a}}
\to  u_{*}\cO_{X_{\Sigma_1}} \to 0 
\]
If we let $f: N_{1} \to N_{2}$ be the injective map of lattices associated to such a morphism, the terms of the resolution are controlled by the compactly supported cohomology of Bondal's stratification of the mirror torus intersected with the subtorus $V:= f^{\vee -1}(0) \subset \T^n$:
\[
\beta_{i,-a} = \dim(H^i_c(S^c_{{[a]}} \cap V )),
\]
where $\beta_{i, -a}$ is the multiplicity of the line bundle $\mathcal{O}_{X_{\Sigma}}(-a)$ in the ${i}$-th term of the resolution. 
\end{theorem}

\begin{example}\label{ex: point min res}
Consider the case where $X_{\Sigma_1}$ is a point $p=(1:1:1)$ and $X_{\Sigma_2} = \P^2$. Bondal's stratification of $\T^2$ for $\P^2$ consists of 3 strata, labeled by line bundles $\mathcal{O}$, $\mathcal{O}(-1)$, and $\mathcal{O}(-2)$. The compactly supported cohomology of $S^c_{0}$ is concentrated in degree 0 and is 1 dimensional, of $S^c_{1}$ is concentrated in degree 1 and is 2 dimensional, and of $S^c_{2}$ is concentrated in degree 2 and is 1 dimensional. Hence, the minimal resolution of $p = (1:1:1) \in \P^2$ must contain one copy of $\mathcal{O}$ in its zeroth term, two copies of $\mathcal{O}(-1)$ in its first term, and one copy of $\mathcal{O}(-2)$ in the second term. Indeed, this describes the terms of the minimal Koszul resolution of the point in $\P^2$. 
\[
    \begin{tikzpicture}

\filldraw[black, thick](-6,-1)circle(2pt);

\filldraw[fill=blue!10!white, draw=black] (-2,-1) -- (0,-1) -- (-2,1)--(-2,-1);

\draw[black, thick](-2,-1)circle(2pt);
\filldraw[white, thick](-2,-1) circle (1pt);

\draw[black, thick](0,-1)circle(2pt);
\filldraw[white, thick](0,-1) circle (1pt);

\draw[black, thick](-2,1)circle(2pt);
\filldraw[white, thick](-2,1) circle (1pt);

\filldraw[fill=red!10!white, draw=black,dashed] (3,1)--(5,1)--(5,-1)--(3,1);

\draw (-6,-0.5) node{$S^c_{0}$};
\draw (-1.3,-0.3) node{$S^c_{1}$};
\draw (4.3,0.3) node{$S^c_{2}$};

    \end{tikzpicture}
\]

\[
    		{\begin{tikzpicture}
            \node (v0) at (-10,1.5) {$0$};
            \node (v1) at (-8,1.5) {$\mathcal{O}(-2)$};
            \node (v2) at (-4.5,1.5) {$\mathcal{O}(-1)^2$};
            \node (v3) at (-1,1.5) {$\mathcal{O}$};
            \node (v4) at (0.5,1.5) {$\mathcal{O}_p$};
            \node (v5) at (2,1.5) {$0$};

            \draw[->]  (v0) edge 
            (v1);            
             \draw[->]  (v1) edge node[above=2pt]{$\begin{pmatrix} x_1 - x_2 \\ x_1 -x_0 \end{pmatrix}$}
            (v2);
            \draw[->]  (v2) edge node[above=2pt]{$\begin{pmatrix} x_0 - x_1 & x_1 - x_2 \end{pmatrix}$} (v3);
            \draw[->]  (v3) edge (v4);
            \draw[->]  (v4) edge
            (v5);
            
  \end{tikzpicture} }
	\]
\end{example}

It is also worth noting that unlike the argument of Hanlon--Hicks--Lazarev for \Cref{main theorem 1}, our approach does not rely on a line bundle resolution of the diagonal. Instead, it coincides with the original argument of Berkesch-Erman-Smith \cite{BES} that uses Beilinson's resolution of the diagonal and hence recovers these minimal resolutions in the case of projective space.  Moreover, our resolution comes with an explicit formula for Betti numbers, allowing one to relate sheaf cohomology to syzygies. We expect that this could be applied to obtain further interesting results relating notions of multigraded Castelnuovo-Mumford regularity and other positivity conditions to syzygies (see \Cref{positivity corrolary}).

\subsection{Outline}
In \S\ref{section 2} we introduce the notion of a stratified space and give  constructions of the cube and CW stratifications of tori coming from toric geometry. We then proceed with describing the category of sheaves constructible with respect to the cube and CW stratifications in terms of quiver path algebras with relations. In \S\ref{Section 3}, we recount some ideas from homological mirror symmetry and specifically the Coherent-Constructible-Correspondence (CCC) along with its functorial properties.  We then relate the category of constructible sheaves from \S\ref{section 2} to the category of constructible sheaves appearing in the CCC. The section culminates in \Cref{theorem: bigdiagram}, relating various categories and serving as the bridge through which we transfer resolutions from topology to algebraic geometry. 
In \S\ref{Section 4} we utilize the results of \S\ref{section 2} and \S\ref{Section 3} to prove \Cref{main theorem 1} and \Cref{main theorem 2}. The first part of \S\ref{Section 4} describes general resolutions; the second is devoted to showing the existence of minimal resolutions of toric subvarieties by line bundles; the third explains how to construct these resolutions along with the ones obtained in \cite{HHL}.

\subsection{Acknowledgement} Our work was greatly influenced by discussions during the \emph{Syzygies and Mirror Symmetry Workshop} at the American Institute for Mathematics.  We are grateful for their hospitality, financial support, and wonderful environment. This project was financially supported by the National Science Foundation DMS \#2302262.   We also thank Christine Berkesch, Andrew Hanlon, and Jesse Huang for helpful discussions. 
\subsection{Notation and conventions}

\begin{itemize}
    \item $\RMod A$ is the category of  right $A$ modules. 
    \item $Q$ is a finite quiver with path algebra $kQ$. $Q_0$ is the set of its verticies and $Q^{ij}_1$ is the edge set between $i$ and $j$. We consider the multiplicative structure defined by
    $h \cdot h'= 0$ unless $t(h) = s(h')$, where t is the target of $h$ and s is $h'$'s source.  
    \item For $A = kQ/R$, where R is an ideal generated by relations $R$, we denote by $P_i = e_iA$ the indecomposable projectives, $I_i = \D(Ae_i)$ the indecomposable injectives, and by $S_i$ the simple modules in the category $\RMod A$.
    \item For any $T \in Ab$, the multiplicative structure on the endomorphism ring $\End(T,T)$ is defined $f \cdot g = f \circ g$
    \item $D_c(X)$ is the bounded derived category of constructible sheaves on $X$; $\Sh(X, \Lambda)$ is the full subcategory of $D_c(X)$ consisting of objects whose singular support lies in $\Lambda$.
\end{itemize}

\section{Poset Stratification of Tori and Exodromy}\label{section 2}
The goal of this section is twofold. First, we introduce the notion of a poset stratified space and give two explicit examples which stratify a real $n$-dimensional torus, called  the cube and CW stratification respectively. One can think of a poset stratification as a decomposition of a topological space into a disjoint union of subspaces (called strata), each labeled by an element of the poset, together with a coherence condition relating the topology of the space with the poset ordering.
In the second part of this section, we turn to the study of the category of sheaves on the real torus that are constructible with respect to cube and cell stratifications. We show that there is a natural collection of generators, which are called probe sheaves, that allow us to identify the constructible category at hand with the category of modules over a certain quiver algebra with relations. We begin with the following definition.

\begin{definition}\label{stratdef} Let $X$ be a topological space. A \newterm{poset stratification} $S = (I,\Phi)$ of $X$ is a continuous map $\Phi: X \rightarrow I$ where $I$ is a poset endowed with the Alexandrov topology (open sets are upward closed).  
\end{definition}
A poset stratification $(I, \Phi)$ 
yields a decomposition of $X$ into the disjoint union of the $\Phi$-level sets $S_\alpha := \Phi^{-1}({\alpha})$.
\[
X = \coprod_{\alpha \in \im(\Phi)} S_\alpha.
\]

\begin{definition}
We call a sheaf $\cF$ on $X$ \newterm{ $S$-constructible} if $\cF|_{S_{\alpha}}$ is a locally constant sheaf of finite rank for all $\alpha \in I$. The (abelian) category of $S$-constructible sheaves on X is denoted $Sh_S(X)$.  We call it the \newterm{$S$-constructible category}.
\end{definition}

\begin{example}\label{ex:cube}
Let $X = \RR^m, I = \ZZ^m$ (with the standard poset structure), and $\Phi = -\floor{\:} $ the floor function.  One checks that this floor function is continuous when $\ZZ^m$ has the Alexandrov topology.  The corresponding stratification $(I,\Phi)$ decomposes $\RR^m$ into half-open cubes:
\[
\RR^m = \coprod_{a \in \mathbb Z^{m}} a+[0,1)^m.
\]
\end{example}

The stratification \Cref{ex:cube} above is crucial in what follows for constructing stratifications appearing in toric mirror symmetry.

\subsection{Cube Stratification of Tori}
\label{subsec: cube}
We now give a construction of a poset stratification of an $n$-dimensional real torus. We produce a poset $I$ and a continuous map $\Phi: \mathbb{T}^n \rightarrow I$ from the stratification of $\R^{n+k}$ defined as in \Cref{ex:cube} and a short exact sequence of lattices: 
 \begin{equation} \label{eq: exact}
     0\rightarrow \mathbb{Z}^n \xrightarrow{\it{i}} \mathbb{Z}^{n+k} \xrightarrow{\pi} \mathbb{Z}^{n+k}/\mathbb{Z}^n \rightarrow 0.
 \end{equation}

Since $\mathbb{R}$ is flat, we can tensor \eqref{eq: exact} by $\mathbb{R}$ to get an exact sequence:
 \begin{equation}
     0\rightarrow \mathbb{R}^n \xrightarrow{\it{i_\mathbb{R}}} \mathbb{R}^{n+k} \xrightarrow{\pi_\mathbb{R}} \mathbb{R}^{n+k}/\mathbb{R}^n \rightarrow 0.
 \end{equation}
To produce this stratification, we require that $\pi_\mathbb{R}(\mathbb{R}^{n+k}_{\geq0})$ (called the effective cone) is strongly convex in $\mathbb{R}^{n+k}/\mathbb{R}^n$ and consists of no non-zero vectors. In toric geometry, this condition is equivalent to requiring that the Picard short exact sequence comes from a proper space. So we assume this in the discussion which follows.
 
Endow $\mathbb{Z}^{n+k}/\mathbb{Z}^n$ with the following poset structure (without the strongly convex assumption above, this relation would only give a preorder):
 
\begin{equation}
[a]:=\pi(a) \geq [b] \: \text{iff} \: [a-b] \in \pi([0,\infty)^{n+k})
\end{equation}

We consider a continuous map:
\begin{align}
\widetilde{\Phi_c}:  \mathbb{R}^n \xrightarrow{\it{i_{\mathbb{R}}}} \mathbb{R}^{n+k} \xrightarrow{-\floor{\:}} \mathbb{Z}^{n+k}\xrightarrow{\pi} \mathbb{Z}^{n+k}/\mathbb{Z}^n
\end{align}

The map $-\floor{\:}$ is the flooring map from \ref{ex:cube}, and $\pi$ is the projection map (which is order preserving and hence continuous in the Alexandrov topology). Consider the natural action of $\mathbb{Z}^n$ on $\mathbb{R}^n$ by translation. It is a free action, and $\widetilde{\Phi}$ is constant on $\mathbb{Z}^n$-orbits, so it descends to a continuous map from the quotient:  

\begin{align}
\Phi_c: \mathbb{T}^n = \mathbb{R}^n / \mathbb{Z}^n \rightarrow \mathbb{Z}^{n+k}/\mathbb{Z}^n
\end{align}

We denote the stratification of $\T^n$ above by $S^c:=(\mathbb{Z}^{n+k}/\mathbb{Z}^n, \Phi_c$). The superscript is motivated by the following proposition, which shows that every stratum of $S^c$ is homeomorphic to an intersection of a half-open cube with the closed subspace that is the image of $\it{i_{\mathbb{R}}}$ in $\mathbb{R}^{n+k}$. Henceforth, we call stratification $S^c$ a cube stratification. 

\begin{proposition}\cite[Lemma 5.4]{HPA}\label{homeobt}
   There exists a canonical homeomorphism $f: S^{c}_{a}:= (-a+[0,1)^{n+k}) \cap \it{i_{\mathbb{R}}(\mathbb{R}}^{n}) \rightarrow S^c_{[a]}$. Furthermore, $\im (\Phi_c) = \pi_{\mathbb{R}}([0,1)^{n+k}) \cap \mathbb{Z}^{n+k}/\mathbb{Z}^n$ (i.e it is a finite set)
\end{proposition}
 
\begin{example}\label{P2strata}
Consider a cube stratification of $\T^2$ coming from the short exact sequence:   
\[
		\scalebox{.75}{\begin{tikzpicture}
            \node (v1) at (-8,1.5) {$0$};
            \node (v2) at (-6.5,1.5) {$\Z^2$};
            \node (v3) at (-2.5,1.5) {$\Z^3$};
            \node (v4) at (1.5,1.5) {$\Z$};
            \node (v5) at (3,1.5) {$0$};
             \draw[->]  (v1) edge
            (v2);
            \draw[->]  (v2) edge node[fill=white]{$\begin{pmatrix} 0 & 1\\ 1 & 0\\ -1& -1\end{pmatrix}$} (v3);
            \draw[->]  (v3) edge node[fill=white]{$\begin{pmatrix} 1& 1& 1 \end{pmatrix}$}(v4);
            \draw[->]  (v4) edge
            (v5);
            
  \end{tikzpicture} }
	\]
The $\im(\Phi_c) = {0,1,2} \in \Z$ with the poset relation $0 < 1 < 2$. Note that the stratum corresponding to the maximal element $2 \in \im (\Phi_c)$, is open and the stratum corresponding to the minimal element, $0$, is closed as expected from the definition of Alexandrov topology.  

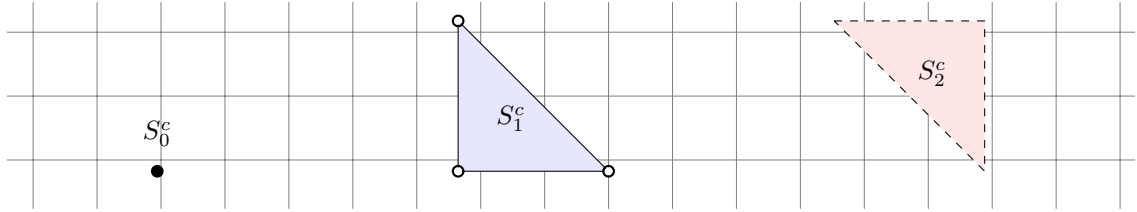
\[
    \begin{tikzpicture}
\draw[step=0.85cm,gray,very thin](-8,-1.5) grid (7,1.25);

\filldraw[black, thick](-6,-1)circle(2pt);

\filldraw[fill=blue!10!white, draw=black] (-2,-1) -- (0,-1) -- (-2,1)--(-2,-1);

\draw[black, thick](-2,-1)circle(2pt);
\filldraw[white, thick](-2,-1) circle (1pt);

\draw[black, thick](0,-1)circle(2pt);
\filldraw[white, thick](0,-1) circle (1pt);

\draw[black, thick](-2,1)circle(2pt);
\filldraw[white, thick](-2,1) circle (1pt);

\filldraw[fill=red!10!white, draw=black,dashed] (3,1)--(5,1)--(5,-1)--(3,1);

\draw (-6,-0.5) node{$S^c_{0}$};
\draw (-1.3,-0.3) node{$S^c_{1}$};
\draw (4.3,0.3) node{$S^c_{2}$};

    \end{tikzpicture}
\]

\end{example}

\subsection{CW Stratification of Tori}
\label{subsec: cell}
We now give another stratification of the torus which refines $S^c$. By \Cref{homeobt} the strata of $S^c$ are homeomorphic to half open cubes intersected with a closed linear subspace. One can decompose a half-open cube into a disjoint union of open cells (subspaces homeomorphic to open balls) and then consider the cell intersections with the linear subspace, giving a more refined stratification. To give some preliminary intuition, the stratification we construct in this section would further decompose the stratum $S^c_1$ from \Cref{P2strata} into four cells: $3$ one-dimensional cells (the sides of the blue triangle) and $1$ two-dimensional open cell (the interior of the blue triangle). As expected, we show this would endow the torus with a CW structure. We now carefully write down the construction of this stratification. 

Consider a CW structure on $\mathbb{R}^m$ where the $0$-skeleton is the set of points with integer coefficients, the $1$-skeleton is the set of unit length open intervals attaching adjacent vertices, the $2$-skeleton is the set of open squares bounded by 4 $1$-cells, etc. The CW cell decomposition of $\mathbb{R}^m$ at hand is the refinement of the decomposition from \Cref{ex:cube}: 

\[ \mathbb{R}^m = 
\coprod_{a \in \mathbb Z^{m}}(a+[0,1)^m) = \coprod_{a \in \mathbb Z^{m}, J \subseteq \{1,..., m\}} \prod_{i \in J} (a+(0, 1)) := \coprod_{a \in \mathbb Z^{m}, J \in 2^{\{1,..., m\}}} e_{(a,J)}
\]

The set of cells in the CW-complex, denoted $I := \mathbb{Z}^m \times 2^{\{1,..., m\}}$ admits a poset structure. There are two ways to define it. Thinking of cells as subsets in $\mathbb{R}^m$, we give the cells the standard CW poset structure where the poset relation coincides with CW boundaries.  Notationally this is given by: \[  (a,J)\geq (a',J') \Leftrightarrow  e_{(a',J')} \subset \overline{e_{(a,J)}}   \]     
Alternatively, thinking of cells as just elements of $I$, we can write:
\begin{equation} \label{eq: cell boudnary}
(a,J) \geq (a',J') \Leftrightarrow {a' = a + \sum_{\ k \in N \subset \{1,..., m\}} e_k, N \subseteq J, J' \subseteq J\setminus N},
\end{equation}
where $e_i$ is a standard basis for $\mathbb{R}^m$.  

In order to make this stratification compatible with the cube stratification, we have to dualize the poset $I$ by introducing a negative sign. Hence, we alter the poset relation on $I$ by declaring:

\[
(a,J): \geq (a',J') \Leftrightarrow (-a,J) \geq (-a',J')
\]

\begin{definition} \label{prop: cell strat}
Let $f: \mathbb{R}^m \rightarrow I $ be a containment map which sends a point x in $\mathbb{R}^m$ to the cell which contains $-x$. Then, $(I,f)$ is a poset stratification of $\mathbb{R}^n$ in a sense of Def. \ref{stratdef}. We call it the \newterm{CW stratification} of $\mathbb{R}^m$.  
\end{definition}

\begin{remark}\label{CWstrat}
Given any CW complex $X$, one can similarly consider its CW stratification 
\[
\Phi: X \to I
\]  
where $I$ is the cell poset of $X$ (defined as \cite[Definition 4.1.1]{Cur}) and $\Phi$ sends any point to the cell in which it is contained.
\end{remark}

We now adapt the construction of a cube stratification of tori to produce a CW stratification. We start with the CW stratification of $\mathbb{R}^{n+k}$ that we constructed earlier in this section together with a short exact sequence of lattices. Again, we require the effective cone to be strongly convex and map $\pi$ to have no zero vectors.  
\begin{equation}
     0\rightarrow \mathbb{Z}^n \xrightarrow{\it{i}} \mathbb{Z}^{n+k} \xrightarrow{\pi} \mathbb{Z}^{n+k}/\mathbb{Z}^n \rightarrow 0.
 \end{equation}

 First, we make sense of   $I/\mathbb{Z}^n$. We consider an action of $\mathbb{Z}^n$ on $I = \mathbb{Z}^{n+k} \times 2^{\{1,..., n\}}$ by first including an element of $\mathbb{Z}^n$ in $\mathbb{Z}^{n+k}$ via $\it{i}$ and then acting by translation on the $\mathbb{Z}^{n+k}$ component of $I$. The quotient $I/\mathbb{Z}^n$ is given a poset structure: 

 \begin{equation}
[(a,J)]:=(\pi(a), J) \geq [(b,J')] \: \text{iff} \: (a,J) \geq (b + m, J')\: \text{for some} \: m \in \ \it{i}(\mathbb{Z}^n)
\end{equation}

We then have a continious map:
\begin{align}
\widetilde{\Phi_{cw}}: \mathbb{R}^n \xrightarrow{\it{i_{\mathbb{R}}}} \mathbb{R}^{n+k} \xrightarrow{f} I \xrightarrow{\pi} I/\mathbb{Z}^n
\end{align} 

Since $\widetilde{\Phi}$ is constant on $\mathbb{Z}^n$-orbits, we have: 
\begin{align}
{\Phi_{cw}}:  \mathbb{T}^n = \mathbb{R}^n/\mathbb{Z}^n \rightarrow I/\mathbb{Z}^n
\end{align}

We denote the stratification of $\mathbb{T}^n$ above by $S^{cw} = (I/\mathbb{Z}^n, \Phi_{cw})$. We have the following analog of \Cref{homeobt} for stratification $S^{cw}$, motivating its name.

\begin{proposition}\label{prop: CW stratification of T^n}
$S^{cw}$ is a CW stratification (in the sense of \Cref{CWstrat}) of a finite CW complex.
\end{proposition}

\begin{proof}
Arguing as in \cite[Lemma 5.4]{HPA}, the universal cover map restricts to a homeomorphism:
\[
\pi: \it{i^{-1}_{\R}}(e_{(a,I)}) \to S^{cw}_{[(a,I)]} 
\]
In other words, every open cell of the CW stratification of $\T^n$ is given by intersecting an open CW cell of $\R^{n+k}$ with the hyperplane arrangement $\it{i_{\mathbb{R}}(\mathbb{R}}^{n})$. Since both sets are convex, their intersection is an open convex set and hence homeomorphic to an open ball. We show that this decomposition is a cellulation of $\T^n$. For every cell, the restriction of $\pi$ extends to the map of closures: 
\[
\pi_{[(a,I)]}: D \cong \overline{\it{i^{-1}_{\R}}(e_{(a,I)})} \to \overline{S^{cw}_{[(a,I)]}}.
\]
Since the closures are compact Hausdorff spaces, the image of $\pi_{[(a,I)]}$ is closed set containing $\it{i^{-1}_{\R}}(e_{(a,I)})$. Hence, it is a continuous surjection. The closure finiteness and weak topology follows from the definition of the stratification. 
\end{proof}

From this point on, whenever we make a statement involving a stratification $S$, we mean either cube or CW stratified torus. We will omit the superscript at $S$ (or $\Phi$) whenever the statement is true for both the cube and CW stratification (or the respective $S$-constructible category). Moreover, in the view of \Cref{prop: CW stratification of T^n}, when working with the stratification $S^{cw}$, we can equivalently think of the stratifying poset $I/\Z^n$ as a cell poset of the CW complex. Hence, we will often use the cell labelling on strata. 

\begin{example}\label{ex: CW stratification of P^2}
The decomposition of $\T^2$ by the strata of its CW stratification from \Cref{P2strata}. 

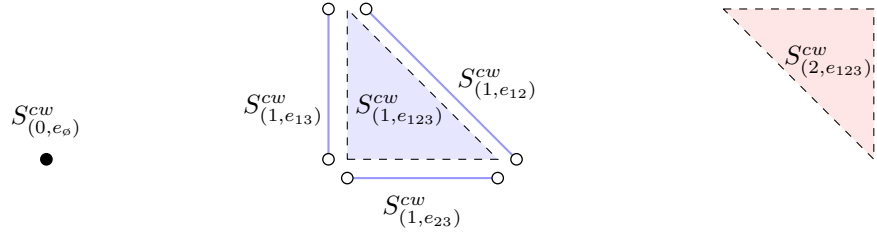
\[
    \begin{tikzpicture}

\filldraw[black, thick](-6,-1)circle(2pt);

\filldraw[fill=blue!10!white, draw=black,dashed] (-2,-1) -- (0,-1) -- (-2,1)--(-2,-1);

\draw[blue!40!white,thick] (-2.25,-1) -- (-2.25,1);

\draw[black, thick](-2.25,-1)circle(2pt);
\filldraw[white, thick](-2.25,-1) circle (1.5pt);

\draw[black, thick](-2.25,1)circle(2pt);
\filldraw[white, thick](-2.25,1) circle (1.5pt);

\draw[blue!40!white,thick] (0.25,-1) -- (-1.75,1);

\draw[black, thick](0.25,-1)circle(2pt);
\filldraw[white, thick](0.25,-1) circle (1.5pt);

\draw[black, thick](-1.75,1)circle(2pt);
\filldraw[white, thick](-1.75,1) circle (1.5pt);

\draw[blue!40!white,thick] (0,-1.25) -- (-2,-1.25);

\draw[black, thick](0,-1.25)circle(2pt);
\filldraw[white, thick](0,-1.25) circle (1.5pt);

\draw[black, thick](-2,-1.25)circle(2pt);
\filldraw[white, thick](-2,-1.25) circle (1.5pt);

\filldraw[fill=red!10!white, draw=black,dashed] (3,1)--(5,1)--(5,-1)--(3,1);

\draw (-6,-0.5) node{$S^{cw}_{(0,e_{\o})}$};
\draw (-1.3,-0.35) node{$S^{cw}_{(1,e_{123})}$};
\draw (-2.85,-0.35) node{$S^{cw}_{(1,e_{13})}$};
\draw (0,0) node{$S^{cw}_{(1,e_{12})}$};
\draw (-1,-1.7) node{$S^{cw}_{(1,e_{23})}$};

\draw (4.4,0.3) node{$S^{cw}_{(2,e_{123})}$};

    \end{tikzpicture}
\]

\end{example}

In defining the cube and CW startification of tori, we also defined the stratifications of their universal covers $\widetilde{S^c}: = (\widetilde{\Phi_c}, \mathbb{Z}^{n+k}/\mathbb{Z}^n)$ and $\widetilde{S^{cw}}: = (\widetilde{\Phi_{cw}}, I/\mathbb{Z}^n)$. We have the following nice property of  $\widetilde{S^{cw}}$.

\begin{proposition}\label{prop: R^n regular CW}
The decomposition of $\R^n$ into connected components of strata of $\widetilde{S^{cw}}$ endows $\R^n$ with the structure of a regular CW complex. 
\end{proposition}
\begin{proof}
We have the following: 
\[
\widetilde{S}_{[(a,I)]}^{cw} \cong \bigsqcup_{b \in \pi^{-1}([a])}{\it{i^{-1}_{\R}}(e_{(b,I)})}
\]
Hence, the connected components of the strata of $\R^{n}$ are homeomorphic to a disjoint union of open balls, indexed by the lifts of $[a]$. Moreover, since $\it{i_{\mathbb{R}}}$ is an inclusion, the cell attachment maps are given by homeomorphisms. Hence, the CW complex is regular.  
\end{proof}

\subsection{Stratifications of Subtori}
Given an $m$-dimensional linear subtorus $V$ of a cube or CW stratified torus, we may consider its induced stratification. Namely, we define its cube stratification to be $S^{V}_{c}:= (\Phi^{c} \circ i, \Z^{n+k}/\Z^{n})$, where $i$ is the inclusion of $V$ in $\T^n$. The strata of $V$ are given by intersections of $V$ with the strata of $\T^n$. The CW stratification is defined analogously. We will denote these stratifications by $S^{V}$.

\begin{example}\label{ex: stratified subtorus}
Consider a subtorus cut out by the equation $2x+3y=0$ inside a cube stratified torus from \Cref{P2strata}. The strata have multiple connected components - the strata $S^{V}_{1}$ and $S^{V}_{2}$ have 2 connected components

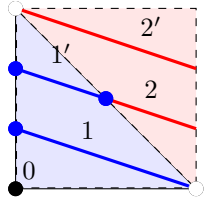
\[\begin{tikzpicture}[scale=1.2]

\filldraw[fill=blue!10!white, draw=black] (-2,-1) -- (0,-1) -- (-2,1)--(-2,-1);

\filldraw[fill=red!10!white, draw=black,dashed] (-2,1)--(0,1)--(0,-1)--(-2,1);

\draw[black,dashed] (-2,-0.99) -- (0,-0.99);
\draw[black,dashed] (-1.99,1) -- (-1.99,-1);

\draw[black, thick](-2,-1)circle(2pt);
\filldraw[black, thick](-2,-1) circle (2pt);

\draw[black, thick](0,-1)circle(2pt);
\filldraw[white, thick](0,-1) circle (2pt);

\draw[black, thick](-2,1)circle(2pt);
\filldraw[white, thick](-2,1) circle (2pt);

\draw[red, very thick] (-1.93, 0.98) -- (0,0.332);
\draw[blue, very thick] (-2, 0.332) -- (-1,0);  
\draw[blue, very thick] (-2, -0.334) -- (-0.08,-0.97); 
\draw[red, very thick] (-1, 0) -- (0,-0.334);

\draw[blue, thick](-2, 0.332)circle(2pt);
\filldraw[blue, thick](-2, 0.332) circle (2pt);

\draw[blue, thick](-2, -0.334)circle(2pt);
\filldraw[blue, thick](-2, -0.334) circle (2pt);
\draw[blue, thick](-1, 0)circle(2pt);
\filldraw[blue, thick](-1, 0) circle (2pt);

\draw (-1.2,-0.35) node{$1$};
\draw (-1.5,0.5) node{$1'$};
\draw (-1.85,-0.8) node{$0$};
\draw (-0.5,0.1) node{$2$};
\draw (-0.5,0.8) node{$2'$};

\end{tikzpicture}
\]
\end{example}

\subsection{Probe and Co-Probe Sheaves}
\label{subsec: probe}
We now turn to the study of the category of sheaves on the real torus that are constructible with respect to the cube and CW stratifications.

First we introduce one of the central notions of the paper - stalk (co)probe sheaves, which are natural generators of $S$-constructible sheaf categories \cite{Zh}. The name is descriptive:

\begin{definition}\label{probe}
We call the representative of the stalk functor at $v \in \mathbb{T}^n$, that is, $\mathcal P_v \in Sh_S(\mathbb{T}^n)$ satisfying
\[
\Hom_{Sh_S(\mathbb{T}^n)}(\mathcal P_v, F) = F_v,\  \forall F\in Sh_S(\mathbb{T}^n)
\]
the \newterm{stalk probe} at $v$.
Similarly, the co-representative of the stalk functor, denoted $\mathcal I_v$, is called a \newterm{co-probe} sheaf.
\[
\Hom_{Sh_S(\mathbb{T}^n)}(F, \mathcal I_v) = (F_v)^{*},\  \forall F\in Sh_S(\mathbb{T}^n)
\]

\end{definition}

\begin{remark}
Since the stalk functor is exact, the probe sheaves are projective objects and co-probe sheaves are injective objects in $Sh_{S}(\mathbb{T}^n)$.
\end{remark}

\begin{remark}\label{lagranprobe}
The same definition of a (co)probe sheaf will be used when we consider, more generally, Lagrangian skeleta in \S~\ref{sub:CCC}.
\end{remark}

We show that for any cube or CW stratified torus, there exists a finite collection of explicitly defined (co)probe sheaves indexed by strata that cover the torus (for any point on the torus, there is a sheaf in this collection which represents the stalk functor at that point). Once we tilt with respect to this collection of sheaves, we are able to identify the category of $S$-constructible sheaves on the torus with the category of modules over a certain path algebra with relations.  

\subsection{Explicit construction of probe and co-probe sheaves}

We first define a notion of entrance/exit paths and entrance/exit spaces for a stratified space. We then give an explicit description of (co)probe sheaves using the entrance/exit spaces. We also give a description of homotopy classes of entrance paths, motivating the notion of an entrance path algebra which appears in the next section. 

\begin{definition}
Let $X$ be a poset stratified space. An \newterm{entrance (resp. exit) path} for is a continuous map
\[
\gamma : [0,1] \to X
\]
such that $\Phi \circ \gamma$ is an order preserving (resp. reversing) map. 
\end{definition}

\begin{remark}\label{rem: HPA has the opposite poset}
In \cite{HPA}, the default topology on the poset is the opposite topology to the Alexandrov topology (see definition 4.4 in \cite{HPA}). Hence, the entrance paths in this paper are exit paths in \cite{HPA} and vice versa. 
\end{remark}

\begin{example}\label{entrancepath}
The images of various entrance paths for the cube stratified torus from \Cref{P2strata}.
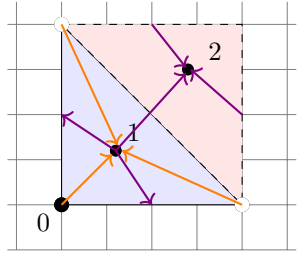
\[
\begin{tikzpicture}[scale=1.2]
\draw[step=0.5cm,gray,very thin](-2.6,-1.5) grid (0.6,1.25);

\filldraw[fill=blue!10!white, draw=black] (-2,-1) -- (0,-1) -- (-2,1)--(-2,-1);

\filldraw[fill=red!10!white, draw=black,dashed] (-2,1)--(0,1)--(0,-1)--(-2,1);

\draw[black, thick](-2,-1)circle(2pt);
\filldraw[black, thick](-2,-1) circle (2pt);

\draw[black, thick](0,-1)circle(2pt);
\filldraw[white, thick](0,-1) circle (2pt);

\draw[black, thick](-2,1)circle(2pt);
\filldraw[white, thick](-2,1) circle (2pt);

\draw (-2.2,-1.2) node{$0$};
\draw (-1.2,-0.2) node{$1$};
\draw (-0.3,0.7) node{$2$};

\draw[black, thick](-1.4,-0.4)circle(1.5pt);
\filldraw[black, thick](-1.4,-0.4) circle (1.5pt);

\draw[black, thick](-0.6,0.5)circle(1.5pt);
\filldraw[black, thick](-0.6,0.5) circle (1.5pt);

\draw[->, orange, thick] (-2,-1) -- (-1.44,-0.44);

\draw[->,orange,thick] (0,-1) -- (-1.33,-0.4);
\draw[->,orange,thick] (-2,1) -- (-1.38,-0.34);
\draw[->, violet, thick] (-1.4,-0.4) -- (-2,0);
\draw[->, violet, thick] (0,0) -- (-0.57,0.5);
\draw[->, violet, thick] (-1.4,-0.4) -- (-1,-1);
\draw[->, violet, thick] (-1,1) -- (-0.62,0.52);
\draw[->, violet, thick] (-1.4,-0.4) -- (-0.62,0.46);
    
    \end{tikzpicture}
\]

\end{example}

\begin{definition}
For a point $y \in X$ we define the \newterm{entrance space} at $y$ to be the subspace
\[
{X_{\Ent}({y}) := \{ x \in X : \exists \: \text{an entrance path} \: {\gamma} \text{ with } {\gamma}(0) = {y}, {\gamma}(1) = x \}\xhookrightarrow{i_y}} X
\]
Similarly, we define the \newterm{exit space} at ${y}$ to be the subspace
\[
X_{\Exit}({y}) := \{ x \in X : \exists \: \text{an exit path} \:  {\gamma} \text{ with } {\gamma}(0) = {y}, {\gamma}(1) = x \}\xhookrightarrow{i_{y}}  Y
\]
\end{definition}

One of the recurring themes of this paper is the idea that a lot of properties of an $S$ stratified torus (or the corresponding $S$-constructible category) can be obtained by understanding the properties of the $\widetilde{S}$-stratification of its universal cover, $\R^n$. In the spirit of this observation, we have the following propositions. 

\begin{example}
Exit spaces for different points in $\R^2$ associated to the cube stratification from \Cref{P2strata}.
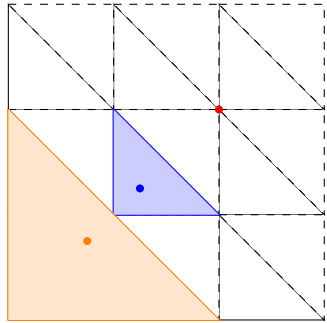
\[
\begin{tikzpicture}[scale=0.7]

\filldraw[fill=white, draw=black] (-2,-1) -- (0,-1) -- (-2,1)--(-2,-1);

\filldraw[fill=white, draw=black,dashed] (-2,1)--(0,1)--(0,-1)--(-2,1);

\filldraw[fill=white, draw=black] (-4,-1) -- (-2,-1) -- (-4,1)--(-4,-1);

\filldraw[fill=white, draw=black,dashed] (-4,1)--(-2,1)--(-2,-1)--(-4,1);

\filldraw[fill=white, draw=black] (-6,-1) -- (-4,-1) -- (-6,1)--(-6,-1);

\filldraw[fill=white, draw=black,dashed] (-6,1)--(-4,1)--(-4,-1)--(-6,1);

\filldraw[fill=white, draw=black] (-2,-3) -- (0,-3) -- (-2,-1)--(-2,-3);

\filldraw[fill=white, draw=black,dashed] (-2,-1)--(0,-1)--(0,-3)--(-2,-1);

\filldraw[fill=white, draw=black] (-4,-3) -- (-2,-3) -- (-4,-1)--(-4,-3);

\filldraw[fill=white, draw=black,dashed] (-4,-1)--(-2,-1)--(-2,-3)--(-4,-1);

\filldraw[fill=white, draw=black] (-6,-3) -- (-4,-3) -- (-6,-1)--(-6,-3);

\filldraw[fill=white, draw=black,dashed] (-6,-1)--(-4,-1)--(-4,-3)--(-6,-1);

\filldraw[fill=white, draw=black] (-2,-5) -- (0,-5) -- (-2,-3)--(-2,-5);

\filldraw[fill=white, draw=black,dashed] (-2,-3)--(0,-3)--(0,-5)--(-2,-3);

\filldraw[fill=white, draw=black] (-4,-5) -- (-2,-5) -- (-4,-3)--(-4,-5);

\filldraw[fill=white, draw=black,dashed] (-4,-3)--(-2,-3)--(-2,-5)--(-4,-3);

\filldraw[fill=white, draw=black] (-6,-5) -- (-4,-5) -- (-6,-3)--(-6,-5);

\filldraw[fill=white, draw=black,dashed] (-6,-3)--(-4,-3)--(-4,-5)--(-6,-3);

\draw[red, thick](-2,-1)circle(1.5pt);
\filldraw[red, thick](-2,-1) circle (1.5pt);

\draw[blue, thick]
(-4,-1) -- (-4,-3) -- (-2,-3) -- (-4,-1);
\fill[blue!20!white, nearly transparent](-4,-1) -- (-4,-3) -- (-2,-3) -- (-4,-1);

\draw[orange, thick]
(-6,-1) -- (-6,-5) -- (-2,-5) -- (-6,-1);
\fill[orange!20!white, nearly transparent](-6,-1) -- (-6,-5) -- (-2,-5) -- (-6,-1);

\draw[blue, thick](-3.5,-2.5)circle(1.5pt);
\filldraw[blue, thick](-3.5,-2.5) circle (1.5pt);
\draw[orange, thick](-4.5,-3.5)circle(1.5pt);
\filldraw[orange, thick](-4.5,-3.5) circle (1.5pt);
\end{tikzpicture}
\]
\end{example}

\begin{example}
Entrance spaces for different points in $\R^2$ associated to the cube stratification $\widetilde{S^c}$ from \Cref{P2strata}.
\[
\begin{tikzpicture}[scale=0.7]

\filldraw[fill=white, draw=black] (-2,-1) -- (0,-1) -- (-2,1)--(-2,-1);

\filldraw[fill=white, draw=black,dashed] (-2,1)--(0,1)--(0,-1)--(-2,1);

\filldraw[fill=white, draw=black] (-4,-1) -- (-2,-1) -- (-4,1)--(-4,-1);

\filldraw[fill=white, draw=black,dashed] (-4,1)--(-2,1)--(-2,-1)--(-4,1);

\filldraw[fill=white, draw=black] (-6,-1) -- (-4,-1) -- (-6,1)--(-6,-1);

\filldraw[fill=white, draw=black,dashed] (-6,1)--(-4,1)--(-4,-1)--(-6,1);

\filldraw[fill=white, draw=black] (-2,-3) -- (0,-3) -- (-2,-1)--(-2,-3);

\filldraw[fill=white, draw=black,dashed] (-2,-1)--(0,-1)--(0,-3)--(-2,-1);

\filldraw[fill=white, draw=black] (-4,-3) -- (-2,-3) -- (-4,-1)--(-4,-3);

\filldraw[fill=white, draw=black,dashed] (-4,-1)--(-2,-1)--(-2,-3)--(-4,-1);

\filldraw[fill=white, draw=black] (-6,-3) -- (-4,-3) -- (-6,-1)--(-6,-3);

\filldraw[fill=white, draw=black,dashed] (-6,-1)--(-4,-1)--(-4,-3)--(-6,-1);

\filldraw[fill=white, draw=black] (-2,-5) -- (0,-5) -- (-2,-3)--(-2,-5);

\filldraw[fill=white, draw=black,dashed] (-2,-3)--(0,-3)--(0,-5)--(-2,-3);

\filldraw[fill=white, draw=black] (-4,-5) -- (-2,-5) -- (-4,-3)--(-4,-5);

\filldraw[fill=white, draw=black,dashed] (-4,-3)--(-2,-3)--(-2,-5)--(-4,-3);

\filldraw[fill=white, draw=black] (-6,-5) -- (-4,-5) -- (-6,-3)--(-6,-5);

\filldraw[fill=white, draw=black,dashed] (-6,-3)--(-4,-3)--(-4,-5)--(-6,-3);

\draw[red, dashed, thick]
(-6,1) -- (0,1) -- (0,-5) -- (-6,1);
\filldraw[red!20!white, nearly transparent](-6,1) -- (0,1) -- (0,-5) -- (-6,1);

\draw[blue, dashed, thick]
(-2,-1) -- (-2,-5) -- (-6,-1) -- (-2,-1);
\filldraw[blue!20!white, nearly transparent](-2,-1) -- (-2,-5) -- (-6,-1) -- (-2,-1);

\draw[orange, dashed, thick]
(-4,-3) -- (-6,-3) -- (-4,-5) -- (-4,-3);
\filldraw[orange!20!white, nearly transparent](-4,-3) -- (-6,-3) -- (-4,-5) -- (-4,-3);

\draw[red, thick](-2,-1)circle(1.5pt);
\filldraw[red, thick](-2,-1) circle (1.5pt);
\draw[orange, thick](-4.5,-3.5)circle(1.5pt);
\filldraw[orange, thick](-4.5,-3.5) circle (1.5pt);
\draw[blue, thick](-3.5,-2.5)circle(1.5pt);
\filldraw[blue, thick](-3.5,-2.5) circle (1.5pt);

\end{tikzpicture}
\]
\end{example}

\begin{proposition}\cite[Proposition 5.5]{HPA}\label{prop: ent/exit are homeo to balls}
For any $\widetilde y \in \widetilde S^{c}_{[a]}$, we have homeomorphisms: 
\[
\mathbb{R}^n_{Exit}(\widetilde y) \simeq \mathbb{R}^{n+k}_{\geq -a}\cap \it{i_{\mathbb{R}}(\mathbb{R}}^{n}) \text{ and }
\mathbb{R}^n_{Ent}(\widetilde y) \simeq \mathbb{R}^{n+k}_{< -a-K}\cap \it{i_{\mathbb{R}}(\mathbb{R}}^{n})
\]
where $K=(-1,...,-1)$. In particular, all entrance spaces are open, contractible, and homeomorphic to an open n-ball, and all exit spaces are closed and contractible.
\end{proposition}

\begin{example}
Entrance spaces for the CW stratification of $\R^n$ from \Cref{P2strata}
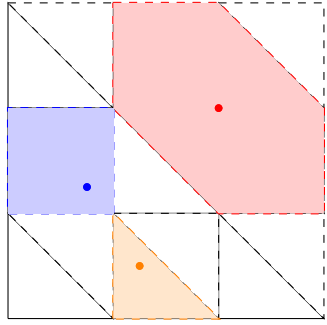
\[
\begin{tikzpicture}[scale=0.7]

\filldraw[fill=white, draw=black] (-2,-1) -- (0,-1) -- (-2,1)--(-2,-1);

\filldraw[fill=white, draw=black,dashed] (-2,1)--(0,1)--(0,-1)--(-2,1);

\filldraw[fill=white, draw=black] (-4,-1) -- (-2,-1) -- (-4,1)--(-4,-1);

\filldraw[fill=white, draw=black,dashed] (-4,1)--(-2,1)--(-2,-1)--(-4,1);

\filldraw[fill=white, draw=black] (-6,-1) -- (-4,-1) -- (-6,1)--(-6,-1);

\filldraw[fill=white, draw=black,dashed] (-6,1)--(-4,1)--(-4,-1)--(-6,1);

\filldraw[fill=white, draw=black] (-2,-3) -- (0,-3) -- (-2,-1)--(-2,-3);

\filldraw[fill=white, draw=black,dashed] (-2,-1)--(0,-1)--(0,-3)--(-2,-1);

\filldraw[fill=white, draw=black] (-4,-3) -- (-2,-3) -- (-4,-1)--(-4,-3);

\filldraw[fill=white, draw=black,dashed] (-4,-1)--(-2,-1)--(-2,-3)--(-4,-1);

\filldraw[fill=white, draw=black] (-6,-3) -- (-4,-3) -- (-6,-1)--(-6,-3);

\filldraw[fill=white, draw=black,dashed] (-6,-1)--(-4,-1)--(-4,-3)--(-6,-1);

\filldraw[fill=white, draw=black] (-2,-5) -- (0,-5) -- (-2,-3)--(-2,-5);

\filldraw[fill=white, draw=black,dashed] (-2,-3)--(0,-3)--(0,-5)--(-2,-3);

\filldraw[fill=white, draw=black] (-4,-5) -- (-2,-5) -- (-4,-3)--(-4,-5);

\filldraw[fill=white, draw=black,dashed] (-4,-3)--(-2,-3)--(-2,-5)--(-4,-3);

\filldraw[fill=white, draw=black] (-6,-5) -- (-4,-5) -- (-6,-3)--(-6,-5);

\filldraw[fill=white, draw=black,dashed] (-6,-3)--(-4,-3)--(-4,-5)--(-6,-3);

\draw[red, dashed, thick]
(-4,1) -- (-2,1) -- (0,-1) -- (0,-3) --(-2,-3) -- (-4,-1) -- (-4,1);
\fill[red!20!white, nearly transparent]
(-4,1) -- (-2,1) -- (0,-1) -- (0,-3) --(-2,-3) -- (-4,-1) -- (-4,1);

\draw[blue, dashed, thick]
(-4,-3) -- (-4,-1) -- (-6,-1) --  (-6,-3) -- (-4,-3);
\fill[blue!20!white, nearly transparent](-4,-3) -- (-4,-1) -- (-6,-1) --  (-6,-3) -- (-4,-3);

\draw[orange, dashed, thick]
(-4,-5) -- (-4,-3) -- (-2,-5) -- (-4,-5);
\fill[orange!20!white, nearly transparent](-4,-5) -- (-4,-3) -- (-2,-5) -- (-4,-5);

\draw[orange, thick](-3.5,-4)circle(1.5pt);
\filldraw[orange, thick](-3.5,-4) circle (1.5pt);
\draw[blue, thick](-4.5,-2.5)circle(1.5pt);
\filldraw[blue, thick](-4.5,-2.5) circle (1.5pt);
\draw[red, thick](-2,-1)circle(1.5pt);
\filldraw[red, thick](-2,-1) circle (1.5pt);
\end{tikzpicture}
\]
\end{example}

\begin{example}
Exit spaces for the CW stratification of $\R^n$ from \Cref{P2strata}
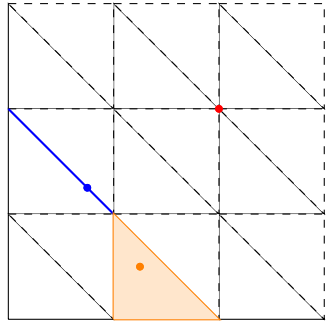
\[
\begin{tikzpicture}[scale=0.7]

\filldraw[fill=white, draw=black] (-2,-1) -- (0,-1) -- (-2,1)--(-2,-1);

\filldraw[fill=white, draw=black,dashed] (-2,1)--(0,1)--(0,-1)--(-2,1);

\filldraw[fill=white, draw=black] (-4,-1) -- (-2,-1) -- (-4,1)--(-4,-1);

\filldraw[fill=white, draw=black,dashed] (-4,1)--(-2,1)--(-2,-1)--(-4,1);

\filldraw[fill=white, draw=black] (-6,-1) -- (-4,-1) -- (-6,1)--(-6,-1);

\filldraw[fill=white, draw=black,dashed] (-6,1)--(-4,1)--(-4,-1)--(-6,1);

\filldraw[fill=white, draw=black] (-2,-3) -- (0,-3) -- (-2,-1)--(-2,-3);

\filldraw[fill=white, draw=black,dashed] (-2,-1)--(0,-1)--(0,-3)--(-2,-1);

\filldraw[fill=white, draw=black] (-4,-3) -- (-2,-3) -- (-4,-1)--(-4,-3);

\filldraw[fill=white, draw=black,dashed] (-4,-1)--(-2,-1)--(-2,-3)--(-4,-1);

\filldraw[fill=white, draw=black] (-6,-3) -- (-4,-3) -- (-6,-1)--(-6,-3);

\filldraw[fill=white, draw=black,dashed] (-6,-1)--(-4,-1)--(-4,-3)--(-6,-1);

\filldraw[fill=white, draw=black] (-2,-5) -- (0,-5) -- (-2,-3)--(-2,-5);

\filldraw[fill=white, draw=black,dashed] (-2,-3)--(0,-3)--(0,-5)--(-2,-3);

\filldraw[fill=white, draw=black] (-4,-5) -- (-2,-5) -- (-4,-3)--(-4,-5);

\filldraw[fill=white, draw=black,dashed] (-4,-3)--(-2,-3)--(-2,-5)--(-4,-3);

\filldraw[fill=white, draw=black] (-6,-5) -- (-4,-5) -- (-6,-3)--(-6,-5);

\filldraw[fill=white, draw=black,dashed] (-6,-3)--(-4,-3)--(-4,-5)--(-6,-3);

\draw[blue, thick]
(-4,-3) -- (-6,-1) ;

\draw[orange, thick]
(-4,-5) -- (-4,-3) -- (-2,-5) -- (-4,-5);
\fill[orange!20!white, nearly transparent](-4,-5) -- (-4,-3) -- (-2,-5) -- (-4,-5);

\draw[orange, thick](-3.5,-4)circle(1.5pt);
\filldraw[orange, thick](-3.5,-4) circle (1.5pt);
\draw[blue, thick](-4.5,-2.5)circle(1.5pt);
\filldraw[blue, thick](-4.5,-2.5) circle (1.5pt);
\draw[red, thick](-2,-1)circle(1.5pt);
\filldraw[red, thick](-2,-1) circle (1.5pt);
\end{tikzpicture}
\]
\end{example}

\begin{proposition}\label{entrance and exit for CW}
For any $\widetilde y \in \widetilde S^{\text{cw}}_{[(a,I)]}$, we have homeomorphisms:     \[
\mathbb{R}^n_{Exit}(\widetilde y) \simeq \overline{e_{(a,I)}} \cap \it{i_{\mathbb{R}}(\mathbb{R}}^{n}) \text{ and } \mathbb{R}^n_{Ent}(\widetilde y) \simeq \{ x \in e_{(b,I')}, \: (b,I') \geq (a,I) \} \cap \it{i_{\mathbb{R}}(\mathbb{R}}^{n}).
\]
In particular, all entrance spaces are homeomorphic to an open $n$-dimensional ball, and all exit spaces are homeomorphic to a closed ball. 
\end{proposition}

\begin{remark}
For a cube or CW stratified $k$-dimensional subtorus $V \subset \T^n$, the entrance and exit spaces are homeomorphic to the intersection of $V$ with the entrance and exit spaces of $\T^n$. In particular, the entrance spaces are homeomorphic to a disjoint union of open $k$-balls, and the exit spaces are homeomorphic to a disjoint union of closed balls. 
\end{remark}

The probe and (co)probe sheaves on $\T^n$ admit a nice description in terms of the entrance and exit spaces of its universal cover. Note \Cref{rem: HPA has the opposite poset}.
    
\begin{proposition}\cite[Proposition 4.13]{HPA} \label{prop: exit/ent sheaves}
Given $y \in \mathbb{T}^n$ choose a lift $\widetilde y \in \mathbb{R}^{n}$. Then,
\begin{align*}
\mathcal P_y & \cong (\pi \circ i_{\widetilde y})_! \mathsf \C_{ \mathbb{R}^n_{\Ent}(\widetilde y)}. \\
\mathcal I_y & \cong (\pi \circ i_{\widetilde{y}})_* \mathsf \C_{ \mathbb{R}^n_{\Exit}(\widetilde y)}.  
\end{align*}
\end{proposition}
\begin{proof}
This is \cite[Proposition 4.13]{HPA} and \cite[Lemma 2.11]{Rouquier}.
\end{proof}

\begin{remark}\label{probes=stratas for torus}
Since the strata of the cube and cell stratification are path connected (this follows from \Cref{homeobt}), an important immediate consequence of \Cref{prop: exit/ent sheaves} is that for all points in the same stratum, the (co)probe sheaf is the same as the corresponding entrance (exit) spaces agree. Henceforth, given a stratum $x \in \im(\Phi)$, we abuse notation and write $\mathcal P_x$ to mean the probe sheaf corresponding to any point in the stratum. On the other hand, the strata of a stratified subtorus $V$ need not have path connected strata. In this case,  we have a  probe sheaf associated to each connected component of each stratum.
\end{remark}

Since the number of strata is finite for both stratifications, we have a finite collection of objects generating the $S$-constructible category by \cite[Theorem 2]{Zh}. In the next section, we show that we can tilt with respect to this collection of probe sheaves, identifying the $S$-constructible category with the category of modules over an endomorphism algebra of this collection. The endomorpshism algebra turns out to admit a nice topological description and is intimately linked to the topology of the stratified torus. Before moving to the next section, we would like to conclude by describing the collection of all entrance paths (up to homotopy) between any two points on $\T^n$. One again, the universal cover plays an important role. 

\begin{proposition}\label{prop:describingent/exit paths}
For any $v \in \T^n$, choose a lift $\widetilde{v} \in \R^n$.
There is a bijection between the set of homotopy classes of entrance paths from $v$ to $w \in \T^n$ and the set of all lifts of $w$ such that $\widetilde{w} \in \R^n_{Ent}(\widetilde{v})$. In particular, since $\R^n_{Ent}(\widetilde{v})$ is compact, the set of homotopy classes of entrance paths between any two points is finite.  
\end{proposition}
\begin{proof}
Since $\R^n$ is the universal cover of $\T^n$, by definition, there is a bijection between the homotopy classes of path between $v$ and $w$ and paths that start at $\widetilde{v}$ and end at some lift of $w$. Since the stratifications of $\T^n$ was obtained by taking a quotient of $\R^n$ and its stratification, the bijection above restricts to the subset of entrance paths. By \Cref{prop: ent/exit are homeo to balls}, the entrance spaces are homeomorphic to open $n$-dimensional balls. Therefore, its intersection with the fiber $\pi^{-1}(w)$ is a finite set. 
\end{proof}

\begin{remark}
The proposition above holds for exit paths if we replace entrance space in the cover with the exit space.
\end{remark}

\subsection{Exodromy for Constructible Sheaves }

The collection of all entrance(exit) paths up to homotopy on a stratified torus forms an algebra that can be viewed as a certain quiver algebra with relations.  

\begin{definition}
Choose an arbitrary base point in every connected component of each stratum. Let $Q_{Ent}$ be a quiver whose vertices are labeled by the connected components of strata and whose arrows are homotopy classes of entrance paths between the base points. Let $R$ be the ideal generated by the homotopy relations between the homotopy classes of paths. Then, the \newterm{entrance paths algebra} $A_{Ent} := kQ_{Ent}/R$. The exit paths algebra is defined analogously. 
\end{definition}

\begin{remark}
Since the intersection of two convex sets is convex, and convex sets in $\R^n$ are contractible, it follows by \Cref{homeobt} that every stratum of the cube and CW stratification of the torus is contractible. Hence, the entrance path algebra is independent of the choice of the base points, and the quiver $Q_{Ent}$ has no loops. Moreover, since the number of homotopy classes of entrance paths between any 2 points is finite (by \Cref{prop:describingent/exit paths}), the entrance paths algebra is finite-dimensional. Note that for subtori, the strata are given by a disjoint union of contractible spaces. Hence, we may have multiple verticies corresponding to the same stratum.  
\end{remark}

\begin{ex}\label{ex: P^2 entrance paths algebra}
The entrance path algebra associated to the cube stratification from \Cref{P2strata}.
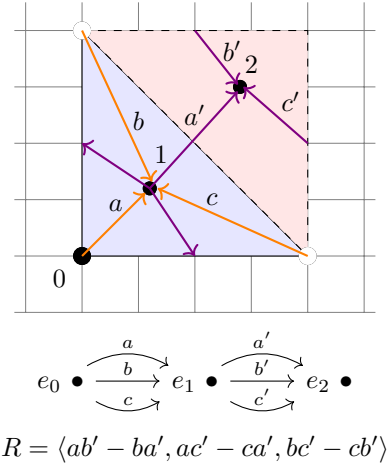
\[
\begin{tikzpicture}[scale=1.5]

\draw[step=0.5cm,gray,very thin](-2.6,-1.5) grid (0.6,1.25);

\filldraw[fill=blue!10!white, draw=black] (-2,-1) -- (0,-1) -- (-2,1)--(-2,-1);

\filldraw[fill=red!10!white, draw=black,dashed] (-2,1)--(0,1)--(0,-1)--(-2,1);

\draw[black, thick](-2,-1)circle(2pt);
\filldraw[black, thick](-2,-1) circle (2pt);

\draw[black, thick](0,-1)circle(2pt);
\filldraw[white, thick](0,-1) circle (2pt);

\draw[black, thick](-2,1)circle(2pt);
\filldraw[white, thick](-2,1) circle (2pt);

\draw (-2.2,-1.2) node{$0$};
\draw (-1.3,-0.1) node{$1$};
\draw (-0.5,0.7) node{$2$};

\draw[black, thick](-1.4,-0.4)circle(1.5pt);
\filldraw[black, thick](-1.4,-0.4) circle (1.5pt);

\draw[black, thick](-0.6,0.5)circle(1.5pt);
\filldraw[black, thick](-0.6,0.5) circle (1.5pt);

\draw[->, orange, thick] (-2,-1) -- (-1.44,-0.44);

\draw[->,orange,thick] (0,-1) -- (-1.33,-0.4);
\draw[->,orange,thick] (-2,1) -- (-1.38,-0.34);
\draw[->, violet, thick] (-1.4,-0.4) -- (-2,0);
\draw[->, violet, thick] (0,0) -- (-0.57,0.5);
\draw[->, violet, thick] (-1.4,-0.4) -- (-1,-1);
\draw[->, violet, thick] (-1,1) -- (-0.62,0.52);
\draw[->, violet, thick] (-1.4,-0.4) -- (-0.62,0.46); 

\draw (-1.7,-0.55) node{$a$};
\draw (-1.5,0.2) node{$b$};
\draw (-0.85,-0.5) node{$c$};

\draw (-1,0.25) node{$a'$};
\draw (-0.15,0.37) node{$c'$};
\draw (-0.67,0.82) node{$b'$};

    \end{tikzpicture}
\]
\[\begin{tikzcd}[scale=1.5]
{{e_0\ \bullet}} & {{e_1\ \bullet}} & {{e_2\ \bullet}}
	\arrow["a",  bend left, from=1-1, to=1-2]
	\arrow["{a'}",  bend left , from=1-2, to=1-3]
	\arrow["c",  bend right, from=1-1, to=1-2]
	\arrow["{c'}",  bend right, from=1-2, to=1-3]
	\arrow["b",  from=1-1, to=1-2]
	\arrow["{b'}", from=1-2, to=1-3]
\end{tikzcd}\]

\begin{equation*}
    R=\langle ab'-ba', ac'-ca', bc'-cb'\rangle
\end{equation*}

\end{ex}

\begin{example}
The exit paths algebra associated to the stratification of the subtorus in \Cref{ex: stratified subtorus}. The verticies of the exit paths quiver correspond to the connected components of the strata. 
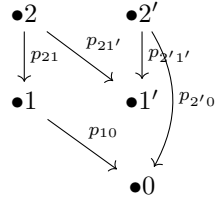
\[\begin{tikzcd}
	{} & \bullet2 & {\bullet 2'} \\
	& {\bullet 1} & {\bullet 1'} \\
	&& {\bullet 0}
	\arrow["p_{21}",from=1-2, to=2-2]
	\arrow["p_{21'}", from=1-2, to=2-3]
	\arrow["p_{2'1'}", from=1-3, to=2-3]
	\arrow["p_{2'0}", bend left, from=1-3, to=3-3]
	\arrow["p_{10}",from=2-2, to=3-3]
\end{tikzcd}\]

\end{example}

\begin{ex}\label{ex:incidenceCW}
Consider the CW stratification 
of a finite, regular CW complex satisfying the axiom of the frontier.  The corresponding entrance path algebra is the incidence algebra of the cell poset. Explicitly the entrance paths are given by boundary maps $\partial_{e_ie_j}$ where $e_i \subseteq \overline{e_j}.$ Note that the CW stratification of $\T^n$ need not be regular. 
\end{ex}

Since both cube and CW stratifications are simple stratifications in a sense of \cite[Definition 4.8]{HPA} by \Cref{prop: ent/exit are homeo to balls} and \Cref{entrance and exit for CW}, we have the following equivalence of categories that we call an exodromy equivalence.

\begin{proposition}\cite[Proposition 4.14]{HPA}\label{exodromy}
There is an equivalence of categories given by the stalk functor:

\begin{align*}
& F: &Sh_{S}(\mathbb{T}^n)  & \rightarrow \RMod \End(\bigoplus_{x\in \im(\Phi)} P_{x}) \\
& &\mathcal{F} & \rightarrow \bigoplus_{x \in \im(\Phi)} \Hom(P_{x}, \mathcal{F}) = \bigoplus_{x \in \im(\Phi)}\mathcal{F}_{v_x} (\text{$v_x$ is the base point of a connected component of stratum $x$)}
\end{align*}

\end{proposition}

\begin{proposition}\label{prop: simpleimages}
The equivalences above take probe sheaves to indecomposable projective modules and co-probe sheaves to indecomposable injective modules. Moreover, it takes constructible sheaves $\mathcal{S}_{x}:= i_{!}\C_{S_{x}}$ to the simple modules $S_{x}$.
\end{proposition}

\begin{proof}
For any $P_{y} \in Sh_{S}(\T^n)$, we have:
\begin{align*}
F(P_{y}) & \cong \bigoplus_{x \in \im{\Phi}}\Hom(\cP_{x},\cP_{y}) & \text{by definition} \\
& \cong e_{y}\End(\bigoplus_{x\in \im(\Phi)} P_{x}) \cong P_{y}
\end{align*}
The argument for co-probe sheaves is the same. For simples, we also argue directly: 
\begin{align*}
\Hom(\cP_y, \mathcal{S}_x) \cong (\mathcal{S}_x)_{y} = \begin{cases}
  \C & \text{if $x = y$} \\
  0 & \text{otherwise}
\end{cases}
\end{align*}
Hence, $F(\mathcal{S}_x) \cong S_{x}$.

\end{proof}

\begin{proposition}\label{prop: entrancepathalgebraisend}
For the cube and CW stratification, we have an isomorphism of rings: 
\begin{align*}\label{prop: end algebra = entrance alg}
\End(\bigoplus_{x \in \im(\Phi)} P_x) \simeq A_{Ent} 
\end{align*}
\end{proposition}

\begin{proof}
First, we show that these are isomorphic as abelian groups.
\begin{align*}
 \Hom(P_x, P_y) & \cong  (P_y)_{x} & \text{$P_x$ is a probe sheaf} \\
& \cong ((\pi \circ i_{\widetilde y})_! \mathsf \C_{ \mathbb{R}^n_{\Ent}(\widetilde y)})_{x} & \text{by \Cref{prop: exit/ent sheaves}} \\
& \cong \bigoplus_{\mathbb{R}^n_{\Ent}(\widetilde y) \cap \pi^{-1}(x)}\C \\
& \cong \bigoplus_{\gamma \in \T^{n}_{\Ent}(y), \gamma(1) = x}\C & \text{by \Cref{prop:describingent/exit paths}}
\end{align*}

We are left to show that the isomorphism above respects the multiplicative ring structure. The elements in $\Hom(P_x, P_y)$ are given by the canonical inclusions of the entrance spaces of $\pi^{-1}(x)$ in the entrance space $\widetilde{y}$ in the universal cover. The isomorphism above takes the entrance path $\gamma$ to the inclusion of the entrance space at $\widetilde{\gamma(1)} = \widetilde{x}$ in $\widetilde{\gamma(0)} = \widetilde{y}$.
The two inclusions are the same if and only if the corresponding entrance paths have the same end point, or equivalently, they are homotopic in the torus. Hence, the composition of morphisms is compatible with the concatenation of entrance paths up to homotopy, which is precisely the multiplicative structure of $A_{Ent}$.
\end{proof}

\begin{proposition}\cite[Corollary 3.6]{Rouquier}\label{prop: proj dim less than n}
The projective dimension of entrance paths algebra $A^{V}_{Ent}$ has: 
\[
pdim(A_{Ent}) \leq k
\]
where $k$ is the dimension of $V \subset \T^n$.
\end{proposition}

\section{Coherent-Constructible Correspondence and Stratified Tori}\label{Section 3}

Let $X_{\Sigma}$ be a smooth, projective, $n$-dimensional toric variety with a complete fan $\Sigma$. Under the non-equivariant Coherent-Constructible Correspondence for toric varieties (CCC), the mirror to the derived category of coherent sheaves on $X$ is the category of constructible sheaves on a torus whose singular support lies in a Lagrangian skeleton $\Lambda_{\Sigma}$ \cite{Tr, Kuw, Zh}. The overarching goal of this section is to understand the relationship between stratified tori, categories of $S$-constructible sheaves, and the CCC.

\subsection{Coherent-Constructible Correspondence and the Cox Construction}\label{sub:CCC}
First, we set up the CCC with a view towards the Cox Construction. The CCC for $X_\Sigma$ is compatible with the Cox construction for $X_\Sigma$ in the following sense.  Cox's quotient construction of $X_\Sigma$ gives: 
\begin{itemize}
\item an open toric subvariety $X_{\widetilde{\Sigma}} = A^{\Sigma(1)} \setminus Z(\Sigma) \subseteq A^{\Sigma(1)}$ where $\widetilde{\Sigma}$ a subfan of the fan of $\mathbb{A}^{\Sigma(1)}$;
\item a toric morphism $\phi: X_{\widetilde{\Sigma}} \to X_{\Sigma}$;
\item an algebraic group $G$ acting on $X_{\widetilde{\Sigma}}$ such that $\phi$ is a geometric quotient for the action of $G$.
\end{itemize}

Since $\widetilde \Sigma$ is a subfan of the standard fan for $\mathbb A^{\Sigma(1)}$, all cones $\widetilde \sigma\in \widetilde \Sigma$ are of the form $\widetilde \sigma_J= \mathbb R_{\geq 0}^J\subset \widetilde N_\mathbb{R}$, with dual cone $\widetilde \sigma_J^\vee=\mathbb R^{J}_{\geq 0} \times \mathbb R^{J^c}$.
Let 
\begin{align*}
\widetilde \Lambda & := \bigcup_{\widetilde \sigma_J \in \widetilde \Sigma, a\in \widetilde M} (\widetilde \sigma_J^\perp +a) \times (\widetilde \sigma_J)\\
& = \bigcup_{\widetilde \sigma_J \in \widetilde \Sigma, a\in \widetilde M} (\mathbb R^{J^c} +a)\times \mathbb R_{\geq 0}^{J}\\
\end{align*}
be the \newterm{$\widetilde{M}$-periodic skeleton} associated with the fan $\widetilde \Sigma$. Let $\pi:\widetilde M_\mathbb{R} \rightarrow \widetilde M_\mathbb R/M$ be the projection, and define

\[
\Lambda_{\widetilde \Sigma} := \widetilde  \Lambda/M.
\]
Then, the \newterm{fiber skeleton}
$$\Lambda_\Sigma:=\Lambda_{\widetilde \Sigma}|_{M_\mathbb{R}/M}$$
is the non-equivariant mirror skeleton to $X_{\Sigma}$. One has an equality: 

$$\Lambda_\Sigma = \bigcup_{\sigma \in \Sigma}( \sigma^{\perp} /\sigma^{\perp} \cap M ) \times (\sigma)$$

\begin{remark}\label{rem:diff signs}
 The skeleton $\Lambda_\Sigma$ we use above agrees with the one in \cite{Zh}.  It    differs from the one in \cite{Tr, Kuw, FLTZ} by a sign convention which have $(-\sigma) \in M_{\R}$.  
\end{remark}

The remark above is related to the following statement comparing different mirror functors appearing in the literature.
\begin{proposition}\label{prop: zhou tr comp}
One has: 
\[
\kappa \cong \mathcal{D} \circ \kappa^{\prime} \circ(-)^{\vee}
\]
where $\mathcal{D}$ is the Verdier duality functor, $(-)^{\vee}:= \Hom(-,\mathcal{O}_{X_{\Sigma}})$, and $\kappa^{\prime}$ is the mirror functor from \cite{Tr}, $\kappa$ is the mirror functor from \cite{Zh}.
\end{proposition}

\begin{proof}
It suffices to show this equality on the generating set of line bundles (which generate since $X_\Sigma$ is smooth). We first recall how $\kappa, \kappa'$ are computed on any line bundle. The procedure is  essentially the same for both. Given any line bundle $\mathcal L$ on $X_\Sigma$, one can associate to it a unique twisted polytope (see \cite[Chapter 3]{Fu}), $\chi_{\mathcal{L}}$. Then, one defines a complex of constructible sheaves associated to a twisted polytope. 

Denote Treumann's complex by $P'(\chi_{\mathcal{L}})$ and Zhou's by $P(\chi_{\mathcal{L}})$.  There are two differences in these complexes: 
\begin{enumerate}
    \item $P'(\chi_{\mathcal{L}})$ uses costandard sheaves whereas $P(\chi_{\mathcal{L}})$ uses standard sheaves;
    \item the twisted polytope associated to $\mathcal L$ by Treumann is the twisted polytope associated to $\mathcal L^\vee$ by Zhou.
\end{enumerate}
In summary,
\begin{equation} \label{eq: Tr vs Zh}
\mathcal D( P'(\chi_{\mathcal{L^\vee}})) =  P(\chi_{\mathcal{L}}) 
\end{equation}

By writing down the \v Cech complex for $\mathcal L$, one has $\kappa'(\mathcal L) \cong P'(\chi_{\mathcal{L}})$. We have:
\begin{align*}
\mathcal{D} \circ \kappa' (\mathcal L^\vee) & = \mathcal D( P'(\chi_{\mathcal{L^\vee}})) & \text{ by definition }  \\
& = P(\chi_{\mathcal{L}}) & \text{By \eqref{eq: Tr vs Zh}} \\
& = \kappa(\mathcal L) & \text{ by definition}
\end{align*}
\end{proof}
\begin{remark}
    Note that Verdier duality $\mathcal{D}$ changes the singular support of a constructible sheaf via the map \cite[Section 7.1]{FLTZ}:
\[
N_{\R} \times M_{R} \to N_{\R} \times M_{R}:(x,y) \to (x,-y).
\]
This is consistent with \Cref{rem:diff signs} and the computation above. 
\end{remark}

The following theorem is the non-equivariant CCC, proved in various contexts, see e.g. \cite{Tr, Kuw, Zh}.  Our notation and conventions coincide with \cite{Zh}. From this point on, we assume that $X_{\Sigma}$ is smooth and projective and has a complete fan. 

\begin{theorem}\label{mainccc}\cite[Theorem 2]{Zh}
Let $X_\Sigma$ be a smooth projective toric variety of complex dimension $n$, then there is an quasi-equivalence of category
\[\kappa: Coh(X_\Sigma) \congto \Sh(\mathbb{T}^n,  \Lambda_\Sigma) \]
where $ \Lambda_\Sigma$ is a conical Lagrangian in $T^*T^n$. 
\end{theorem}

\subsection{Stratified Tori and Toric Geometry}

Since $X_{\Sigma}$ is an $n$-dimensional smooth projective toric variety with a complete fan, the short exact sequence for the Picard group of $X_{\Sigma}$ allows us to construct both a cube and CW stratification of  $\mathbb{T}^n$ as in \S\ref{subsec: cube} and \S\ref{subsec: cell}.

In \S\ref{subsec: probe}, we considered a collection of sheaves (called probe sheaves) on the real torus associated with its cube and CW stratification. Recall that, to each stratum of the cube or CW stratification, we associated a probe sheaf. Since every stratum is labeled by an element of the poset, and for cube stratification the poset is a Picard group, we automatically have a correspondence between strata, probe sheaves, and line bundles. This correspondence is, in fact, a manifestation of mirror symmetry. More precisely, we have the following proposition:
\begin{proposition} \cite[Prop. 5.5]{HPA} \label{prop: ring iso}
There is a ring isomorphism:
\begin{align*}
\alpha: A^{c}_{Ent} \cong \End(\bigoplus_{[a] \in \im(\Phi_c)} P_{[a]}) \to \End ( \bigoplus_{[a] \in \im(\Phi_c)} \mathcal O_{X_{\Sigma}}(-a) \ )
\end{align*}
\end{proposition}

\begin{example}
The short exact sequence for the Picard group of $\P^2$ is given in \Cref{P2strata}. By \Cref{prop: ring iso} and \Cref{prop: entrancepathalgebraisend}, the endomorphism algebra of the direct sum of line bundles $\mathcal{O}(-2), \mathcal{O}(-1), \mathcal{O}$ is isomorphic to the quiver path algebra with relations from \Cref{entrancepath}. Moreover, the entrance paths $a$ and $a'$ correspond to the morphism $x_{0}$, $b$ and $b'$ to the morpshism $x_2$, and $c$ and $c'$ to $x_1$. Thus, the ideal of homotopy relations on the homotopy classes of paths on the torus becomes the ideal of commutativity relations on $x_0$, $x_1$, and $x_2$ i.e.\ $\langle x_0x_1 - x_1x_0, x_0x_2 - x_2x_0, x_1x_2 - x_2x_1 \rangle$.
\end{example}

\subsection{Establishing the Categorical Relationship}\label{sub: categorical relationship}

To connect the CCC to stratified tori, we first define Lagrangian skeleta coming from the cube and CW stratification:

\begin{align*}
\Lambda_c & := \bigcup_{ [a] \in \im(\Phi_c)} ss(S^{c}_{[a]})
\end{align*}

\begin{align*}
\Lambda_{cw} & := \bigcup_{ [(a,J)] \in \im(\Phi_{cw})} ss(\mathcal{S}^{cw}_{[(a,J)]})
\end{align*}
where $ss(\mathcal{S}_{\alpha})$ stands for a singular support of sheaf $i_{!}\C_{S_{\alpha}}$.

The following two propositions will be used to relate the categories appearing in the CCC to categories of S-constructible sheaves.
\begin{proposition}\label{prop:comparing skeleta}
There are containments of Lagrangian skeleta
\[
\Lambda_{\Sigma} \subseteq \Lambda_{c} \subseteq \Lambda_{cw}
\]
which induces inclusions of full subcategories
\[
\Sh(\mathbb{T}^n,  \Lambda_\Sigma) \subseteq 
\Sh(\mathbb{T}^n, \Lambda_{c}) \subseteq \Sh(\mathbb{T}^n, \Lambda_{cw}).
\]
\end{proposition}

\begin{proof}
Observe that
\[
\Lambda_{c} = \Lambda_{\widetilde{\Sigma}_{\A^n}}|_{M_\mathbb{R}/M}
\]
Since $\widetilde{\Sigma} \subset \widetilde{\Sigma}_{\A^n}$, the first inclusion follows. The second inclusion is just by definition as the CW stratification is a refinement of the cube stratification.
\end{proof}

\begin{proposition}\label{derexod}
The following categories are equivalent: 
\begin{align*}
\Sh(\mathbb{T}^n, \Lambda_{c}) \cong D(Sh_{S^c}(\mathbb{T}^n)) \cong D(mod-A^c_{Ent})
\end{align*}
\begin{align*}
\Sh(\mathbb{T}^n, \Lambda_{cw}) \cong D(Sh_{S^{cw}}(\mathbb{T}^n)) \cong D(mod-A^{cw}_{Ent})
\end{align*}
\end{proposition}

\begin{proof}
The proofs of these statements are more general and only use the fact that the stratifications $S^c, S^{cw}$ which satisfy the requirements of \Cref{prop: exit/ent sheaves} and \Cref{exodromy}.  We give the proof for $S^c$.

The equivalence  $D(Sh_{S^c}(\T^n)) \cong D(mod-A^c_{Ent})$ is just \Cref{exodromy}.  Since  $D(mod-A^c_{Ent})$ is generated by the free module $A^c_{Ent} = \bigoplus_{[a] \in \im \Phi^c} P_{[a]}$. It follows from \Cref{prop: simpleimages} that  $D(Sh_{S^c}(\T^n))$ is generated by the probe sheaves $\mathcal P_{x} \in Sh_{S^c}(\T^n)$.  

Now, the fully-faithful functor 
\[
inc: D(Sh_{S^{c}}(\T^n)) \to D_c(\T^n)
\]
lands in $\Sh(\T^n, \Lambda_c)$ since $ss(\mathcal P_{x}) \subseteq \Lambda_{c}$ by \Cref{prop: exit/ent sheaves}. On the other hand, $\mathcal P_{x}$ is the stalk probe at $x$ by \cite[Lemma 2.11]{Rouquier} and $\Sh(\T^n, \Lambda_c)$ is generated by the probes \cite[Theorem 2]{Zh}. 
 Hence, the fully-faithful inclusion is an equivalence.
\end{proof}

\begin{remark}
The categories of sheaves singularly supported on the cube and CW skeleta are in some sense nicer to work with because they admit a collection of generators that are pure objects and so we can construct resolutions by these generators in classical sense.  These resolutions are then eventually transferred through mirror symmetry to produce resolutions in toric geometry.  
\end{remark}

\begin{remark}As the notion of a (co)probe sheaf makes sense in the derived setting of category $\Sh(\T,  \Lambda_\Sigma)$ (see \Cref{lagranprobe}), and the collection of (co)probe sheaves automatically generates the category, toric  Homological Mirror Symmetry itself can be proven by comparing (co)probe sheaves and their mirrors.  This is the approach by Zhou to prove \Cref{mainccc}.  Indeed, he demonstrates that the image of the collection of line bundles $\mathcal{O}_{X_{\Sigma}}(-a)\otimes{\omega^{-1}_{X_{\Sigma}}}, [a] \in \im{\Phi}$ under $\kappa$ are probe sheaves (up to a homological shift) and compares their dg endomorphism algebra (a generalization of \Cref{prop: ring iso}). Notably, the probe sheaves are not necessarily pure sheaves but rather a complex of constructible sheaves associated to a line bundle called a twisted polytope sheaf.
\end{remark}

We are finally ready to establish the main results relating the three categories:
\[
\Sh(\mathbb{T}^n,  \Lambda_\Sigma) \subseteq 
\Sh(\mathbb{T}^n, \Lambda_{c}) \subseteq \Sh(\mathbb{T}^n, \Lambda_{cw}).
\]. 

If $X$ is a smooth manifold and $\Lambda'\subset T^*X$ is a Lagrangian skeleton with a closed subskeleton $\Lambda\subset \Lambda'$, there are two functors (called wrapping functors)
$$
\rho^L: \Sh(X,  \Lambda_{\Sigma'})\rightarrow \Sh(X,  \Lambda_\Sigma)
$$
$$
\rho^R: \Sh(X,  \Lambda_{\Sigma'})\rightarrow \Sh(X,  \Lambda_\Sigma)
$$
that are left and right adjoints respectively to the inclusion $\rho: \Sh(X,  \Lambda_\Sigma) \rightarrow \Sh(X,  \Lambda_{\Sigma'})$ (see \cite{Kuo}). Where necessary, we will put a subscript $\rho_{\Lambda}^R$ to indicate the skeleton we are wrapping to.

A very important property of functors $\rho_L$ and $\rho_R$ is that they behave nicely with respect to (co)probe sheaves.

\begin{lemma}\cite[Lemma 2.1]{Rouquier} \label{lem: probe to probe}
Let $\Lambda \subseteq \Lambda'$ be a closed subskeleton, then
\[
\rho^{L}(\mathcal P_v^{\Lambda'}) = \mathcal P_v^\Lambda
\]
\[
\rho^{R}(\mathcal I_v^{\Lambda'}) = \mathcal I_v^\Lambda
\]
\end{lemma}

\begin{remark}
One should be careful when using the right wrapping functor $\rho^{R}$ since it generally does not preserve compact objects. This becomes relevant when one works with toric varieties that are not smooth and complete (see \cite[Theorem 1.2]{Kuw}).
\end{remark}

The following diagrams are our main tools for transferring resolutions through mirror symmetry.  
\begin{theorem}\label{theorem:big diagram1}
Let $T = \bigoplus_{[a] \in \im(\Phi_c)} \mathcal{O}_{X_\Sigma}(-a)$ and $A = \End(T) \cong A^{c}_{Ent}$ by \Cref{prop: ring iso}. Then, the following diagram commutes:
\[ 
\begin{tikzcd}
 D(Coh(X_\Sigma)) & D(\RMod A^{c}_{Ent}) \arrow[shift right, swap]{l}{{-}\otimes_{A}{T}} \arrow{d}{\rotatebox{90} \: F^{-1}} \\
\Sh(\T^n, \Lambda_{\Sigma}) \arrow{u}{\rotatebox{90} \: S  \circ \kappa^{-1}} & D(Sh_{S^c}(\mathbb{T}^n)) \arrow[shift right, swap]{l}{\rho_{\Lambda}^L}
\end{tikzcd}
\]
\end{theorem}

\begin{proof}
We check commutativity on the idecomposable projectives of $\RMod A^{c}_{Ent}$, which generate $D(\RMod A^{c}_{Ent})$. 

\begin{align*}
& S \circ \kappa^{-1} \circ \rho_{\Lambda}^L \circ F^{-1} (P_{[a]}) \cong S \circ \kappa^{-1} \circ \rho_{\Lambda}^L (\mathcal P_{[a]}) & \text{by \Cref{prop: simpleimages}} \\
& \cong S \circ \kappa^{-1}(\mathcal P_{[a]}) & \text{left wrapping takes probe sheaves to probe sheaves by \Cref{lem: probe to probe}} \\ 
& \cong \mathcal{O}_{X_{\Sigma}}(-a) & \text{by \cite[Proposition 4.2]{Zh}} \\
& \cong P_{[a]}\otimes_{A}{T}
\end{align*}
\end{proof}

\begin{theorem}\label{theorem: bigdiagram}
The following diagram commutes:
\[ 
\begin{tikzcd}
 D(Coh(X_\Sigma)) & D(\RMod A_{Ent}) \arrow[shift right, swap]{l}{{-}\otimes_{A}{T}} \arrow{d}{\rotatebox{90} \: F^{-1} \circ S} \\
\Sh(\T^n, \Lambda_{\Sigma}) \arrow{u}{\rotatebox{90} \: \kappa^{-1}} & D(Sh_{S^c}(\mathbb{T}^n)) \arrow[shift right, swap]{l}{\rho_{\Lambda}^R}
\end{tikzcd}
\]
\end{theorem}

\begin{proof}

We check commutativity on the idecomposable projectives of $\RMod A^{c}_{Ent}$, which generate $D(\RMod A^{c}_{Ent})$.

\begin{align*}
& \kappa^{-1} \circ \rho_{\Lambda}^R \circ F^{-1} \circ S (P_{[a]}) \cong \kappa^{-1} \circ \rho_{\Lambda}^R \circ F^{-1} (I_{[a]}) & \text{by the properties of Serre functor} \\
& \cong \kappa^{-1} \circ \rho_{\Lambda}^R(\mathcal{I}_{[a]}) & \text{by \Cref{prop: simpleimages}} \\
& \cong \kappa^{-1}(\mathcal{I}_{[a]}) & \text{by \Cref{lem: probe to probe}} \\ 
& \cong \mathcal{O}_{X_{\Sigma}}(-a) & \text{by \cite[Proposition 4.2]{Zh}} \\
& \cong P_{[a]}\otimes_{A}{T} \\
\end{align*}    
\end{proof}

The following corollary is a surprising consequence as explained originally by Favero and Huang and to appear in forthcoming work:
\begin{corollary}\label{cor: rhom is fully faithful}
The functor $RHom(T, -):  D(Coh(X_\Sigma)) \rightarrow D(\RMod A)$ is fully-faithful.   
\end{corollary}
\begin{proof}
In \Cref{theorem:big diagram1}, passing to the right adjoints, we have:
\[
RHom(T, -) \cong F \circ \rho  \circ \kappa \circ S^{-1} (-)
\]
Since the inclusion $\rho$ is fully faithful and the rest of functors are equivalances, the result follows.  
\end{proof}

\subsection{Algebraic descriptions of wrapping functors}

It may appear somewhat mysterious what the inclusion functor and its left and right adjoints mean geometrically. In the case when the skeleta are coming from the cube and/or CW stratifications, by \Cref{exodromy} and \Cref{derexod}, the categories $\Sh(\mathbb{T}^n, \Lambda_{c})$ and $\Sh(\mathbb{T}^n, \Lambda_{cw})$ are equivalent to the derived category of modules over an entrance path algebra for the cube and CW stratification respectively. When that is the case, we show that inclusion and the (left) wrapping functor is simply a pushforward (forgetful functor) and pullback (base change) pair associated to a certain homomorphism between the entrance path algebras. 

Since the cube stratification $S^c$ is a coarsening of the CW stratification $S^{cw}$, we have:

\begin{proposition}\label{cellcubeposetmap}
There is an order-preserving map of posets $f: I/\mathbb{Z}^n \rightarrow \mathbb{Z}^{n+k}/\mathbb{Z}^n$, inducing a ring homomorphism 
\begin{align*}
\phi: A^{\text{cw}}_{Ent} & \rightarrow A^{c}_{Ent} \\
\gamma & \mapsto \gamma
\end{align*}
\end{proposition}

\begin{proof}
The map
\begin{align*}
f: I/\Z^n & \to \Z^{n+k}/\Z^{n} \\
[(a,J)] & \mapsto [a]
\end{align*}
is order preserving since 

\begin{align*}
& [(a,J)] \leq [(b,J')] \implies -a - m = -b + \sum_{\ k \in N \subset \{1,..., n+k\}} e_k \text{ for some} \: m \in \it{i}(\mathbb{Z}^n)
\end{align*}

Hence, $b-a \in \pi([0,\infty])^{n+k}$, which, by definition, means that $[a]\leq[b]$. The fact that $f$ is a ring homomorphism is immediate from the definition.
\end{proof}

Under the equivalence of \Cref{exodromy} pushforward under $\phi$ corresponds to inclusion and (derived) pullback corresponds to the left wrapping.  Precisely, we have the following:
\begin{proposition}\label{prop: CW and cube funct}
The following diagram commutes. In particular, $i_{*}$ admits a left adjoint which preserves probe sheaves, and a right adjoint which preserves co-probe sheaves.

\[ 
\begin{tikzcd}
Sh_{S^{c}}(\mathbb{T}^n) \arrow{r}{i_{*}} \arrow{d}{\rotatebox{90}{$\sim$} \: F_c} & Sh_{S^{\text{cw}}}(\mathbb{T}^n) \arrow{d}{\rotatebox{90}{$\sim$} \: F_{cw}} \\
\RMod A^{c}_{Ent} \arrow{r}{\phi_{*} = for} & \RMod A^{\text{cw}}_{Ent} 
\end{tikzcd}
\]
\end{proposition}

\begin{proof}
We first note:
\begin{align}\label{eq: projectives}
& \phi^{*}(P_{([a],J)}) \cong  P_{[(a,J)]}\otimes_{A^{\text{cw}}_{Ent}} A^{\text{c}}_{Ent} \notag \\
& \cong \phi(e_{[a,J]})A^{c}_{Ent} \notag  \\
& \cong e_{[a]}A^{c}_{Ent} \cong P_{[a]}
\end{align}

Then, for any $\cF \in Sh_{S^c}(\T^n)$, we have:
\small
\begin{align*}
\Hom(A^{cw}_{Ent}, \phi_{*} \circ F_{c}(\F))  
& \cong \bigoplus_{([a,J]) \in \im(\Phi^{cw})} \Hom(\phi^{*}P_{([a,J])}, F_{c}(\F)) & \text{by adjunction} \\
& \cong \bigoplus_{([a,J]) \in \im(\Phi^{cw})} \Hom(P_{[a]}, F_{c}(\F)) & \text{by \eqref{eq: projectives}} \\
& \cong \bigoplus_{([a,J]) \in \im(\Phi^{cw})} \Hom(\cP_{[a]}, \F) & \text{applying $F_c^{-1}$ and \Cref{prop: simpleimages}} \\
& \cong \bigoplus_{([a,J]) \in \im(\Phi^{cw})} \F_{a} & \text{by definition of a stalk probe} \\
& \cong \bigoplus_{([a,J]) \in \im(\Phi^{cw})} \Hom(\cP_{([a,J])}, i_{*}(\F)) & \text{again by definition of a stalk probe}  \\ 
& \cong \Hom(A^{cw}_{Ent}, F_{cw} \circ i_{*} (\F))  & \text{applying the equivalence $F_{cw}$ and \Cref{prop: simpleimages}}   
\end{align*}
\normalsize
Therefore the diagram commutes by the Yoneda lemma.

Hence, the left adjoint to $inc$ is given by 
\[
inc^{L} \cong F^{-1}_{c} \circ \phi^{*} \circ F_{cw}
\]
and satisfies: 
\begin{align*}
& inc^{L}(\cP_{([a,J])}) \cong F^{-1}_{c} \circ \phi^{*} \circ F_{cw}(\cP_{([a,J])}) \\
& \cong F^{-1}_{c} \circ \phi^{*} (\cP_{([a,J])}) \cong F^{-1}_{c}(P_{[a]}) \cong \cP_{[a]} & \text{by \eqref{eq: projectives} and \Cref{prop: simpleimages}}
\end{align*}

Since $A^{c}_{Ent}$ and $A^{cw}_{Ent}$ are finite dimensional algebras of finite homological dimension, they admit Serre functors. Hence, the formal right adjoint to $i_{*}$ exists, and since Serre functors take indecomposable projectives to injectives and commute with exact equivalences, the adjoint takes co-probe sheaves to co-probe sheaves. 

\end{proof}

\section{Line Bundle Resolutions from Mirror Symmetry}\label{Section 4}

In this section, we begin to utilize the observations made in \S\ref{sub: categorical relationship} to describe sheaf resolutions on smooth projective toric varieties by line bundles. We first make a general statement about resolutions of any coherent sheaf and then move to more specific examples coming from pushforwards by toric morphisms. Following \cite{BHS}, we begin by defining the notion of a minimal resolution.

\begin{definition}
A line bundle resolution of a coherent sheaf $F$ on a smooth projective toric variety $X$, denoted  $\mathcal{O}^{\bullet}_{X_{\Sigma}}$, is called minimal if no differential contains an identity morphism. A minimal resolution is called irreducible if it is not isomorphic to a chain complex of the form $\mathcal{O}^{\bullet'}_{X_{\Sigma}} \oplus \mathcal{O}^{\bullet''}_{X_{\Sigma}}$, where $\mathcal{O}^{\bullet''}_{X_{\Sigma}}$ is acyclic.
\end{definition}

\begin{example}
Consider a minimal Koszul complex in $x_{0}$ and $x_2$ on the Hirzerbruch surface $\mathbb F_2$. Since $V(x_0,x_2)$ is in the unstable locus, the complex is acyclic.  Hence, it is not irreducible. 
\end{example}

As remarked in \cite{BHS}, unlike free resolutions in commutative algebra, minimal resolutions of coherent sheaves by line bundles need not be unique and can have different lengths. However, if the inclusion of the Lagrangian skeleta in \Cref{prop:comparing skeleta} is an equality: 
\[
\Lambda_{\Sigma} = \Lambda_{c}
\]
then the minimal resolution of $F$ constructed in \Cref{theorem: any resolution exists} is unique (and hence irreducible) among resolutions by line bundles in $\Phi_c$. This is the statement of \Cref{prop: unique resolution}.

\begin{remark}
For the projective space $\P^{n}$ and products of projective spaces, $\Lambda_{\Sigma} = \Lambda_{c}$.   
\end{remark}

\subsection{General Resolutions}\label{sub: general resolutions}

\begin{theorem}\label{theorem: any resolution exists}
Let $F$ be a coherent sheaf on a smooth, $n$-dimensional projective toric variety $X_\Sigma$ and $\mathcal{L}$ be an ample line bundle. Then, $F$ admits a minimal resolution of length at most $n$ by line bundles in the collection ${\mathcal O_{X_\Sigma}(-a) \otimes \mathcal{L} ^{-N}, [a] \in \im{\Phi_c}}$, for some fixed (sufficiently high) $N$.
\end{theorem}

\begin{proof}
By Serre's vanishing theorem for projective varieties, for any ample line bundle $\mathcal{L}$ and a sufficiently high $N$, 
\[
H^{p}(T, F \otimes \mathcal{L} ^{N}) = 0 
\]
for $p \geq 0$, where $T = \bigoplus_{[a] \in \im(\Phi_c)} \mathcal{O}_{X_\Sigma}(-a)$. Hence,
\[
RHom(T,F \otimes \mathcal{L} ^{N}) \cong M 
\] for some $M \in \RMod A_{Ent}$. By \Cref{prop: proj dim less than n} , $pdim(A_{Ent})\leq n$. Hence, the minimal projective resolution of $M$, denoted $P^{\bullet}$, has length at most $n$. Therefore,
\begin{align*}
 F \otimes \mathcal{L} ^{N} & \cong RHom(T,F \otimes \mathcal{L} ^{N})\otimes_{A}{T} & \text{since $RHom(T,-)$ is fully faithful by \Cref{cor: rhom is fully faithful}} \\
& \cong P^{\bullet} \otimes_{A}{T} \\
& \cong \mathcal{O}^{\bullet}_{X_\Sigma}(-a) & \text{since $-\otimes_{A}{T}$ takes indecomposable projectives to line bundles}
\end{align*}
Twisting the resolution by $\mathcal{L}^{-N}$, we get:
\[
F \cong \mathcal{O}^{\bullet}_{X_\Sigma}(-a) \otimes \mathcal{L}^{-N}
\]
Since minimal projective resolutions in the category of graded finite-dimensional quiver algebras are characterized by having no identity components in the differentials, the resolution $F$ is minimal.
\end{proof}

\begin{proposition}\label{prop: unique resolution}
Suppose a smooth, projective toric variety $X_{\Sigma}$ satisfies: 
\[
\Lambda_{\Sigma} = \Lambda_{c}
\]
Then, the minimal resolution of $F$ by line bundles in the collection ${\mathcal O_{X_\Sigma}(-a) \otimes \mathcal{L} ^{-N}, [a] \in \im{\Phi_c}}$ from \Cref{theorem: any resolution exists} is unique.  
\end{proposition}
\begin{proof}
If $\Lambda_{\Sigma} = \Lambda_{c}$, then $\Sh(\T^n,\Lambda_{\Sigma}) = \Sh(\T^n,\Lambda_{c})$. Hence, by the commutativity of \Cref{theorem: bigdiagram} and since $\kappa$, $F$, and $S$ are equivalences, $- \otimes_{A}T$ is an equivalence, and so is its right adjoint $RHom(T,-)$. Since any minimal projective resolution in the category of finite dimensional algebras is unique and irreducible, the claim follows. 
\end{proof}

\subsection{Betti Numbers for General Resolutions}

Given a simple module $S_{[a]} \in \RMod A^{c}_{Ent}$, we may consider its minimal injective  resolution, which we denote by $\mathcal{I}^{\bullet}_{[a]}$. Then, applying the functor $S^{-1}(-) \otimes_{A}T$ to $\mathcal{I}^{\bullet}_{[a]}$ yields a complex of line bundles, which we denote by $\mathcal{O}^{\bullet}_{[a]}$.

\begin{proposition}\label{prop: general betti numbers}
The Betti numbers for the resolution of $F$ from \Cref{theorem: any resolution exists} satisfy:
\begin{align*}
& \beta^{c}_{i,[a]} = dim(H^{i}(Hom^\bullet(\mathcal{O}^{\bullet}_{[a]}, F \otimes \mathcal{L}^{N})))
\end{align*}
where $\beta^{c}_{i,[a]}$ is the multiplicity of the line bundle $\mathcal{O}(-a)$ in the $i^{th}$ term of the resolution.
\end{proposition}
\begin{proof}
Recall that 
\[
RHom(T,F \otimes \mathcal{L} ^{N}) \cong M 
\] for some $M \in \RMod A_{Ent}$. The resolution of $F$ in \Cref{theorem: any resolution exists} was obtained by applying $ - \otimes_{A}T$ to the minimal projective resolution of $M$. Hence, its Betti numbers satisfy:
\begin{align*}
\beta^{c}_{i,[a]} & = \dim (\Hom_{D(\RMod A_{Ent})}(M, S_{[a]}[i])) \\
& = \dim (\Hom_{D(\RMod A_{Ent})}(S_{[a]},S(M)[i])) & \text{S is a Serre functor} \\
& = \dim (\Hom_{D(\RMod A_{Ent})}(S^{-1}S_{[a]},M[i])) & \text{S is an equivalence} \\
& = \dim (\Hom_{D(\RMod A_{Ent})}((S^{-1}S_{[a]})\otimes_{A}T,F\otimes \mathcal{L} ^{N}[i])) & \text{$RHom(T,-)$ is fully faithful} \\
& = \dim (\Hom_{D(\RMod A_{Ent})}(\mathcal{O}^{\bullet}_{[a]},F\otimes \mathcal{L} ^{N}[i]))\\
& = \dim(H^{i}(Hom^\bullet(\mathcal{O}^{\bullet}_{[a]}, F \otimes \mathcal{L}^{N}))) & \text{by \cite[Lemma 5.9, p.89]{GKZ}}
\end{align*}
\end{proof}

\begin{corollary}\label{positivity corrolary}
Suppose a coherent sheaf $F$ satisfies:  

\begin{align*}
H^{i}(X_{\Sigma}, F(a)) = 0 & \text{ for $i > 0$ and all $a \in \im{\Phi_c}$}
\end{align*}
Then, $F$ admits a minimal resolution by line bundles in $\im(\Phi_c)$ of length bounded by the dimension of $X_{\Sigma}$ with the following Betti numbers:
\begin{align*}
& \beta^{c}_{i,[a]} = dim(H^{i}(Hom^\bullet(\mathcal{O}^{\bullet}_{[a]}, F)))
\end{align*}
\end{corollary}

\begin{proof}
This follows immediately from \Cref{theorem: any resolution exists} and \Cref{prop: general betti numbers}.
\end{proof}

\subsection{Resolving Pushforward Sheaves}

We now turn to resolving specific sheaves on $X_\Sigma$ that are obtained from pushforwards of structure sheaves along nice toric morphisms. We will see that the resolutions of these coherent sheaves are intimately related to the topology of the torus together with its cube and CW stratifications. The functoriality of homological mirror symmetry will play a crucial role. 

Consider a pair of smooth projective toric varieties $X_{\Sigma_1}$ and $X_{\Sigma_2}$ with complete fans $\Sigma_1$ and $\Sigma_2$. We will only consider finite toric morphisms.  Recall that a finite toric morphism comes from an injective map $f:N_1 \to N_2$ such that $\Sigma_1 = f^{-1}_{\mathbb R}(\Sigma_2)$.
We denote the dual map by $f^\vee: M_2 \to M_1$.

To align notation with \cite{Tr}, we use $v$ to denote the dual map of real vector spaces $v:= f^\vee_{\RR}: M_{2,\R} \to M_{1,\R}$, which also descends to the map of tori: 
\[
\overline{v}: M_{2,\R}/M_{2} \to M_{1,\R}/M_{1}.
\]
We denote the zero fiber by $\newterm{V} := \overline{v}^{-1}(0)$. 
If $f$ is in addition injective, the CCC enjoys nice functorial properties \cite[Theorem 2.5]{Tr}. This functoriality is the main ingredient in computing the mirror of the pushforward of the structure sheaf of $X_{\Sigma_1}$ by $u$. Once we obtain its mirror, using \Cref{theorem: bigdiagram}, we will be able to associate to it a pure module over the entrance path algebra, whose minimal projective resolution would yield a resolution of $u_{*}\mathcal{O}_{X_{\Sigma_1}}$ by line bundles in $\im(\Phi^c)$.

\begin{remark}
 In what follows, the only thing we use is that the map $u$ comes from an injective fan-preserving map of lattices.  This actually is equivalent to asking that $u$ is a generically finite toric morphism. However, such a morphism always factors as a birational toric morphism followed by a finite toric morphism.  Hence, the pushforward $u_* \mathcal O_{\Sigma_1}$ will always be isomorphic to the pushforward along a toric finite morphism.  Hence, we phrase things in this simpler language.      
\end{remark}

First, we establish the functoriality of $\kappa$, heavily bootstrapping of the functoriality of $\kappa'$ (see \Cref{prop: zhou tr comp}). 

\begin{proposition}\label{zhou's functoriality}
In the situation above, one has:

\begin{equation*}
   \begin{tikzcd}
DCoh(X_{\Sigma_2}) \ar[r, "\kappa_2"] \ar[d, "u^{*}"]& Sh(\T^n, \Lambda_{\Sigma_2}) \ar[d, "\overline{v}_{!}"] \\ DCoh(X_{\Sigma_1}) \ar[r, "\kappa_1"]  & Sh(\T^n, \Lambda_{\Sigma_1})
\end{tikzcd}
\label{eq: commutativity}
\end{equation*}
where $\kappa$ is the mirror functor in \cite{Zh}.
\end{proposition}

\begin{proof}
Since $\kappa \cong \mathcal{D} \circ \kappa^{\prime} \circ(-)^{\vee}$ by \Cref{prop: zhou tr comp}, the claim will follow by showing the commutativity of the following diagram:
\begin{equation*}
   \begin{tikzcd}
DCoh(X_{\Sigma_2}) \ar[r, "(-)_2^{\vee}"] \ar[d, "u^{*}"]& DCoh(X_{\Sigma_2})^{op} \ar[r, "\kappa'^{op}_2"] \ar[d, "u^{*}"] & \Sh^{op}(\T^n, \Lambda_{-\Sigma_2}) \ar[r, "\mathcal{D}_2"] \ar[d, "\overline{v}_{!}"]& Sh(\T^n, \Lambda_{\Sigma_2}) \ar[d, "\overline{v}_{!}"] \\ DCoh(X_{\Sigma_1}) \ar[r, "(-)_1^{\vee}"] & DCoh(X_{\Sigma_1})^{op} \ar[r, "\kappa'^{op}_1"]  & \Sh^{op}(\T^n, \Lambda_{-\Sigma_1}) \ar[r, "\mathcal{D}_1"]  & Sh(\T^n, \Lambda_{\Sigma_1})
\end{tikzcd}
\label{eq: commutativity}
\end{equation*}
where $\kappa'_{i}$ is the mirror functor in \cite{Tr}. The commutativity of the middle square is the functoriality of the Treumann's mirror functor \cite[Proposition 2.5]{Tr}. The commutativity of the left-most square is a standard fact about locally-free sheaves. The commutativity of the right-most square can be argued as follows: 

\begin{align*}
&\overline{v}_{!} \circ \mathcal{D}_2 \cong \mathcal{D}_1 \circ \overline{v}_{*} & \text{by the properties of Verdier duality} \\ 
& \cong \mathcal{D}_1 \circ \overline{v}_{!}  & \text{since $\overline{v}$ is proper}
\end{align*}

\end{proof}

\begin{proposition}\label{prop: imagekappa}
One has:
\[
\kappa_2(u_{*}(\mathcal{O}_{X_{\Sigma_1}})) = \rho^{R}_{\Lambda_{\Sigma}}(\overline{v}^{!}j_{\{0\}*}\C_{\{0\}}).
\]
\end{proposition}

\begin{proof}

Consider the functoriality diagram of \Cref{zhou's functoriality} where we take the right adjoints to the vertical pair of functors. Namely, the right adjoint of $u^{*}$ is $u_{*}$ since $u$ is a proper morphism, and the right adjoint to $\overline{v}_{!}$ is $\rho^R_{\Lambda_{\Sigma_2}}\circ \overline{v}^{!}$. 

\begin{equation}
   \begin{tikzcd}
DCoh(X_{\Sigma_2}) \ar[r, "\kappa_2"] & (\T^n, \Lambda_{\Sigma_2})  \\ DCoh(X_{\Sigma_1}) \ar[u, "u_{*}"] \ar[r, "\kappa_1"]  & (\T^n, \Lambda_{\Sigma_1}) \ar[u,  "\rho^R_{\Lambda_{\Sigma_2}}\circ \overline{v}^{!}"']
\end{tikzcd}
\label{eq: commutativity adjoint}
\end{equation}
The diagram commutes since the original diagram does and the horizontal functors $\kappa_1$ and $\kappa_2$ are equivalences. We now can compute: 
\begin{align*}
&\kappa_2 (u_{*}\mathcal{O}_{\Sigma_1}) \cong \rho^R_{\Lambda_{\Sigma_2}} \circ \overline{v}^{!}(\kappa_1(\mathcal{O}_{\Sigma_1})) &
\text{by the commutativity of \eqref{eq: commutativity adjoint}}\\ 
& \cong \rho^{R}_{\Lambda_{\Sigma_2}}(\overline{v}^{!}j_{\{0\}*}\C_{\{0\}}) & \text{since $\kappa_1(\mathcal{O}_{\Sigma_1}) \cong j_{\{0\}*}\C_{\{0\}}$ by \cite[Proposition 3.12]{Zh}}.
\end{align*}

\end{proof}

We now show that there is a module $M_{\Sigma_1} \in \RMod A_{Ent}$, which once viewed as an element in $D(\RMod A_{Ent})$ in degree 0 has:

\begin{align*}
M_{\Sigma_1}\otimes_{A}{T} \cong u_{*}\mathcal{O}_{\Sigma_1}
\end{align*}
where $T = \bigoplus_{[a] \in \im(\Phi_c)} \mathcal{O}_{X_\Sigma}(-a)$ (see \Cref{theorem: bigdiagram}).

\begin{proposition}\label{prop:object to resolve}
Let $M_{\Sigma_1} = S^{-1} \circ F ( \rho^{R}_{\Lambda_c}(\overline{v}^{!}j_{\{0\}*}\C_{\{0\}}))$. Then, $M_{\Sigma_1}\otimes_{A}{T} \cong u_{*}\mathcal{O}_{\Sigma_1}$, where $S, F$ and $T$ are as in \Cref{theorem: bigdiagram}. Moreover, $M_{\Sigma_1}$ is quasi-isomorphic to a module over the entrance path algebra  $A_{Ent}$ which is viewed as an element in $D(\RMod A_{Ent})$ in degree 0.  
\end{proposition}

\begin{proof}

\begin{align*}
M_{\Sigma_1}\otimes_{A}{T}
& \cong \kappa_2^{-1} \circ \rho^{R}_{\Lambda_{{\Sigma_2}}} \circ F^{-1} \circ S (M_{\Sigma_1}) & \text{by \Cref{theorem: bigdiagram}} \\
& \cong \kappa_2^{-1} \circ \rho^{R}_{\Lambda_{{\Sigma_2}}} (\rho^{R}_{\Lambda_{c2}}(\overline{v}^{!}j_{\{0\}*}\C_{\{0\}})) & \text{by definition of $M_{\Sigma_1}$} \\
& \cong \kappa_2^{-1} \circ \rho^{R}_{\Lambda_{{\Sigma_2}}} (\overline{v}^{!}j_{\{0\}*}\C_{\{0\}}) & \text{since $\Lambda_{\Sigma} \subseteq \Lambda_{c}$ by \Cref{prop:comparing skeleta}} \\
& \cong u_{*}\mathcal{O}_{\Sigma_1} & \text{by \Cref{prop: imagekappa}}
\end{align*}

We check that $M_{\Sigma_1}$ is quasi-isomorphic to a pure module by checking that the dimension vector at any vertex $[a]$, $M_{\Sigma_1[a]}$, is concentrated in degree 0 (recall that the verticies of the quiver $Q_{Ent}$ are labeled by the strata of the cube stratification, which are in turn labeled by poset elements):

\begin{align*} 
& M_{\Sigma_1[a]} = \\
& = \Hom_{D(\RMod A_{Ent})}(P_{[a]}, M_{\Sigma_1}) & \text{since $P_{[a]}$ is an indecomposable projective}\\
& = \Hom_{D(\RMod A_{Ent})}(S(P_{[a]}), S(M_{\Sigma_1})) & \text{since Serre functor is an equivalence} \\
& = \Hom_{D(\RMod A_{Ent})}(I_{[a]}, F( \rho^{R}_{\Lambda_c}(\overline{v}^{!}j_{\{0\}*}\C_{\{0\}}))) & \text{as $S$ takes  projectives to injectives} \\
& = \Hom_{\Sh(\mathbb{T}^n, \Lambda_{c})}(\mathcal{I}_{[a]}, \rho^{R}_{\Lambda_c}(\overline{v}^{!}j_{\{0\}*}\C_{\{0\}})) & \text{$F^{-1}$ is an equivalence and \Cref{prop: simpleimages}}\\
& = \Hom_{D_c(\T^n)}(\mathcal{I}_{[a]}, \overline{v}^{!}j_{\{0\}*}\C_{\{0\}}) & \text{right wrapping is the right adjoint to the inclusion}\\
& = \Hom_{D_c(\T^n)}((\pi \circ i_{\widetilde x})_{*} \mathsf \C_{ \mathbb{R}^n_{\Exit}(\widetilde x)}, \overline{v}^{!}j_{\{0\}*}\C_{\{0\}}) & \text{where $x$ is any point in $S^c_{[a]}$; follows by \Cref{prop: exit/ent sheaves}}\\
\end{align*}

To proceed, consider the following two cartesian squares fitting into one:
\begin{equation}
\begin{tikzcd}
\mathbb{R}^n_{\Exit}(\widetilde x) \cap (\pi\circ\overline{v})^{-1}(0) \ar[r, "i_s"] \ar[d, "g"]& \mathbb{R}^n_{\Exit}(\widetilde x) \ar[d, "t:= \pi \circ i_{\widetilde{x}}"] \\
\overline{v}^{-1}(0) \ar[r, "h"] \ar[d, "p"]& \T^n \ar[d, "\overline{v}"] \\ {0} \ar[r, "j_{\{0\}}"] & \T^m
\end{tikzcd}
\label{eq: fiber product}
\end{equation}

From these Cartesian diagrams of locally compact spaces, we obtain proper base change isomorphisms:
\begin{align}
\overline{v}^{!} j_{\{0\}*}(\F) \cong h_{*}p^{!}(\F) \label{eq: bc1}\\
t^{!} h_{*}(\G) \cong i_{s*}g^{!}(\G) \label{eq: bc2}
\end{align}

So continuing the sequence of equalities above we have:
\begin{align*} 
\mathcal F_x & = \Hom_{D_c(\T^n)}(t_! \mathsf \C_{ \mathbb{R}^n_{\Exit}(\widetilde x)},h_{*}p^{!}(\C_{\{0\}})) & \text{$t_! = t_{*}$ since $t$ is proper by \Cref{prop: ent/exit are homeo to balls} and \eqref{eq: bc1}} \\
& = \Hom_{D_c(\mathbb{R}^n_{\Exit}(\widetilde x))}( \C_{ \mathbb{R}^n_{\Exit}(\widetilde x)},t^{!}h_{*}p^{!}(\C_{\{0\}})) & \text{by adjunction} \\
& = \Hom_{D_c(\mathbb{R}^n_{\Exit}(\widetilde x))}( \C_{ \mathbb{R}^n_{\Exit}(\widetilde x)},i_{s*}g^{!}p^{!}(\C_{\{0\}})) & \text{by \eqref{eq: bc2}}\\
& = \Hom_{D_c(\{0\})}((p\circ g)_{!}\C_{ \mathbb{R}^n_{\Exit}(\widetilde x) \cap (\pi\circ\overline{v})^{-1}(0)},\C_{\{0\}}) & \text{by adjunction} \\
& = H_{c}(\mathbb{R}^n_{\Exit}(\widetilde x) \cap (\pi\circ\overline{v})^{-1}(0))^{*} & \text{ since $p \circ g$ is the map to a point} \\
\end{align*}

By \Cref{prop: ent/exit are homeo to balls}, the exit spaces are given by closed polytopes. The fiber $(\pi\circ\overline{v})^{-1}(0)$ is isomorphic to a disjoint union of closed $k$ dimensional subspaces. Thus, its intersection with the exit space is a closed convex set, which must be homeomorphic to a closed ball in $\R^n$. Its compactly supported cohomology is therefore concentrated in degree $0$ and is $p$-dimensional, where $p$ is the number of its connected components. This finishes the proof.
\end{proof}

\begin{remark}
The above proposition is a bit subtle. Since by \Cref{prop:object to resolve}, we know that the $\kappa$ image of $u_{*}\cO_{X_{\Sigma_1}}$ is $\rho^{R}_{\Lambda_{\Sigma_2}} (\overline{v}^{!}j_{\{0\}*}\C_{\{0\}})$, naively, we could have tried to directly resolve its image $S^{-1} \circ F ( \rho^{R}_{\Lambda_{\Sigma_2}}(\overline{v}^{!}j_{\{0\}*}\C_{\{0\}}))$ in $D(\RMod A_{Ent})$. However, it need not be a pure module (and in fact it rarely is). Therefore, we don't quite understand how to resolve the mirror of $u_{*}\cO_{X_{\Sigma_1}}$. Instead we resolve a module which (viewed as a constructible sheaf under the exodromy equivalence up to a Serre twist) wraps to the mirror of $u_{*}\cO_{X_{\Sigma_1}}$.
\end{remark}

\begin{theorem}\label{theorem: min res of subvarieties}
Let $u : X_{\Sigma_1} \to X_{\Sigma_2}$ be finite toric morphism of smooth projective toric varieties.  The sheaf $u_{*}\cO_{X_{\Sigma_1}}$ admits a minimal line bundle resolution of length $k$ by sums of line bundles in the collection ${\mathcal O_{X_\Sigma}(-a), [a] \in \im{\Phi_c}}$
\[
0 \to \bigoplus_{[a] \in \im{\Phi_c}} \cO(-a)^{\oplus \beta_{k,-a}}
\to ... \to  \bigoplus_{[a] \in \im{\Phi_c}} \cO(-a)^{\oplus \beta_{0,-a}}
\to  u_{*}\cO_{X_{\Sigma_1}} \to 0 
\]
The terms of the resolution are controlled by compactly supported cohomology of the strata of the cube stratification intersected with $V$: 
\[
\beta_{i,-a} = \dim(H^i_c(S^c_{{[a]}} \cap V )),
\]
where $\beta_{i, -a}$ is a multiplicity of the line bundle $\mathcal{O}_{X_{\Sigma}}(-a)$ in the ${i}$-th term of the resolution. Furthermore, if 
\[
\Lambda_{\Sigma} = \Lambda_{c}
\]
then the minimal resolution is unique.
\end{theorem}

\begin{proof}
Since by \Cref{prop:object to resolve},  \begin{align*}
M_{\Sigma_1}\otimes_{A}{T} \cong u_{*}\mathcal{O}_{\Sigma_1}
\end{align*}
given a minimal projective resolution of $M_{\Sigma_1}$ by indecomposible projectives $P_{[a]}$, applying $\otimes_{A}{T}$, we obtain a minimal projective resolution of $u_{*}\mathcal{O}_{\Sigma_1}$ by line bundles $\mathcal O_{X_\Sigma}(-a)$. Since the algebra $A_{Ent}$ is finite dimensional (this follows from \Cref{prop:describingent/exit paths}), $M_{\Sigma_1}$ admits a unique minimal projective resolution. The multiplicities $\beta_{i, -a}$ of its terms can be expressed using Ext groups:
\begin{align}
\beta_{i,-a} = \dim(\Ext^i(M_{\Sigma_1},S_{[a]})) = \dim (H^{i}R\Hom(M_{\Sigma_1},S_{[a]}))
\end{align}

Using the exodromy equivalence, we can compute these multiplicities topologically. Consider the following Cartesian squares, where $S^c_a$ defined as in \Cref{homeobt}:
\begin{equation}
   \begin{tikzcd}
S^c_a \cap (\pi\circ\overline{v})^{-1}(0) \ar[r, "i_a"] \ar[d, "g"]& S^c_a \ar[d, "t:= \pi \circ i_{\widetilde{a}}"] \\
\overline{v}^{-1}(0) \ar[r, "h"] \ar[d, "p"]& \T^n \ar[d, "\overline{v}"] \\ {0} \ar[r, "j_{\{0\}}"] & \T^m
\end{tikzcd}
\label{eq: fiber product}
\end{equation}

\begin{align*}
& \Hom_{D(\RMod A_{Ent})}(M_{\Sigma_1},S_{[a]})=\\
& = \Hom_{D(\RMod A_{Ent})}(S^{-1} \circ F ( \rho^{R}_{\Lambda_c}(\overline{v}^{!}j_{\{0\}*}\C_{\{0\}})),S_{[a]})\\
& = \Hom_{D(\RMod A_{Ent})}(F ( \rho^{R}_{\Lambda_c}(\overline{v}^{!}j_{\{0\}*}\C_{\{0\}})),S(S_{[a]})) & \text{Serre functor is an equivalence}\\
& = \Hom_{D(\RMod A_{Ent})}(S_{[a]},F( \rho^{R}_{\Lambda_c}(\overline{v}^{!}j_{\{0\}*}\C_{\{0\}})))^{*} & \text{by the property of Serre functor}\\
& = \Hom_{\Sh(\mathbb{T}^n, \Lambda_{c})}(i_{!}\C_{S^c_{[a]}},\rho^{R}_{\Lambda_c}(\overline{v}^{!}j_{\{0\}*}\C_{\{0\}}))^{*} & \text{F is an equivalence and \Cref{prop: simpleimages}} \\
& = \Hom_{\Sh(\mathbb{T}^n, \Lambda_{c})}(t_! \mathsf \C_{ S^c_{a}},\rho^{R}_{\Lambda_c}(\overline{v}^{!}j_{\{0\}*}\C_{\{0\}}))^{*} & \text{by \Cref{homeobt}} \\
& = H_{c}(S^c_{a} \cap (\pi\circ\overline{v})^{-1}(0)) & \text{by the same argument as in \Cref{prop:object to resolve}} \\
& = H_{c}(S^c_{[a]} \cap V) & \text{since $\pi$ is a homeomorphism for strata by \Cref{homeobt}} \\
\end{align*}

Therefore, 

\begin{align*}
& \beta_{i,-a} = \dim(H^{i}_{c}(S_{[a]} \cap V))
\end{align*}

We now argue that the length of the minimal projective resolution is always equal to $k$.  Consider the stratification $S_V^{c}:= (\Phi^{V}_c, \mathbb{Z}^{n+k}/\mathbb{Z}^n)$ of $V \subset \T^n$, where 

\begin{align}\label{cube strat of V}
\Phi^{V}_c := \Phi_c \circ h: V \subset \T^n \to \mathbb{Z}^{n+k}/\mathbb{Z}^n
\end{align}
Since the $\im(\Phi_c)$ is finite by \Cref{homeobt}, so is $\im({\Phi^{V}_c})$. In particular, it must contain a maximal element $[a]_{max}$. Since the set $\{[a]: [a] \geq [a]_{max} \}$ is open in $\mathbb{Z}^{n+k}/\mathbb{Z}^n$, its preimage $S^{V}_{[a]_{max}}$ is open.

Now, there is a homeomorphism:
\begin{align*}
S^{V}_{[a]} \cong S^c_{[a]} \cap \overline{v}^{-1}(0) \cong S^c_{a} \cap (\pi\circ\overline{v})^{-1}(0)
\end{align*}
Since $(\pi\circ\overline{v})^{-1}(0)$ is homeomorphic to a disjoint union of closed $k$-dimensional subspaces, $S^{V}_{[a]}$ is homeomorphic to an open convex set in a disjoint union $\coprod_{w \in \mathbb Z^m} \RR^k$. Therefore, it is homeomorphic to a disjoint union of open $k$-dimensional balls. Since, the compactly supported cohomology of an open $k$-dimensional ball is concentrated in degree $k$, we have:
\begin{align*}
& \beta_{k,-a_{max}} = \dim(H^{k}_{c}(S^V_{[a]_{max}} )) > 0
\end{align*}

Finally, we show that the length of the resolution is at most $k$. Since the stratification $S^V$ is a coarsening of the CW stratification, $S^{V}_{[a]}$ is a union of open cells of dimension at most $k$.  Now let $S^{V}_{[a], k-1}$ be the union of all cells of dimension at most $k-1$.  This is a closed subset and by induction we may assume it has compactly supported cohomology in degree at most $k-1$.  The relative long exact sequence gives the result since an open cell has compactly supported cohomology concentrated in degree $k$. 

If \[
\Lambda_{\Sigma} = \Lambda_{c}
\], then $\Sh(\T^n,\Lambda_{\Sigma}) = \Sh(\T^n,\Lambda_{c})$. Hence, by the commutativity of \Cref{theorem: bigdiagram} and since $\kappa$, $F$, and $S$ are equivalences, $- \otimes_{A}T$ is an equivalence, and so is its right adjoint $RHom(T,-)$. Since any minimal projective resolution in the category of finite dimensional algebras is unique and irreducible, the claim follows.

\end{proof}

\begin{remark}\label{another interpretation of multiplicities}
The multiplicities of the line bundles appearing in the resolution of $u_{*}\cO_{X_{\Sigma_1}}$ can be given another topological interpretation. One has:
\begin{align*}
& H_{c}(S^c_{[a]} \cap V) \\
& = H_{c}(S^c_{a} \cap (\pi\circ\overline{v})^{-1}(0)) \\
& = \Hom_{D_c(\{0\})}((p\circ g)_{!}\C_{ S^c_a \cap (\pi\circ\overline{v})^{-1}(0)},\C_{\{0\}})^{*} \\
& = \Hom_{D_c(S^c_a \cap (\pi\circ\overline{v})^{-1}(0))}(\C_{ S^c_a \cap (\pi\circ\overline{v})^{-1}(0)},(p\circ g)^{!}\C_{\{0\}})^{*} \\
& = \Hom_{D_c(S^c_a \cap (\pi\circ\overline{v})^{-1}(0))}(\C_{ S^c_a \cap (\pi\circ\overline{v})^{-1}(0)},\mathcal{D}(\C_{ S^c_a \cap (\pi\circ\overline{v})^{-1}}(0)))^{*} \\
& = H^{BM}(S^c_a \cap (\pi\circ\overline{v})^{-1}(0))^{*} \\
& = H^{BM}(S^c_{[a]} \cap V)^{*}
\end{align*}
Since the closure of $S^c_{[a]} \cap V$ is given by a closed polytope, we can write $S^c_{[a]} \cap V = Y - Y'$, where $Y$ is homeomorphic to a compact simplicial complex and $Y'$ is a closed subcomplex. Then, the Borel-Moore homology of $S^c_{[a]} \cap V$ can be computed as a relative homology of the pair  $H(Y,Y')$. 
\end{remark}

\begin{example}\label{ex: point min res}
Consider the case where $X_{\Sigma_1}$ is a point. We think of $N_1 = 0$ as the zero lattice including into $N_2$. By \Cref{theorem: min res of subvarieties} and since $V = \T^n$ (so that the intersection with $V$ is trivial), the compactly supported cohomology of the strata control the terms of the minimal resolution of $\mathcal{O}_p$. Consider the case when $X_{\Sigma_2} = \P^2$. Then, the compactly supported cohomology of $S^c_{0}$ is concentrated in degree 0 and is 1 dimensional, of $S^c_{1}$ is concentrated in degree 1 and is 2 dimensional, and of $S^c_{2}$ is concentrated in degree 2 and is 1 dimensional. Hence, the minimal resolution of $p = (1:1:1) \in \P^2$ must contain one copy of $\mathcal{O}$ in its zeroth term, two copies of $\mathcal{O}(-1)$ in its first term, and one copy of $\mathcal{O}(-2)$ in the second term. Indeed, this describes the terms of the minimal Koszul resolution of the point in $\P^2$. 
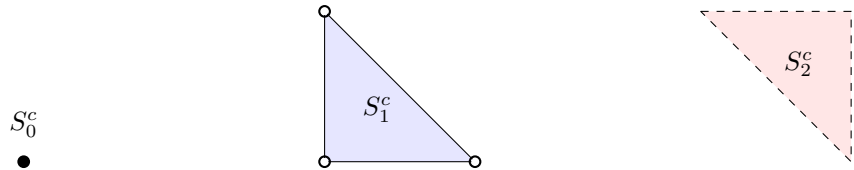
\[
    \begin{tikzpicture}

\filldraw[black, thick](-6,-1)circle(2pt);

\filldraw[fill=blue!10!white, draw=black] (-2,-1) -- (0,-1) -- (-2,1)--(-2,-1);

\draw[black, thick](-2,-1)circle(2pt);
\filldraw[white, thick](-2,-1) circle (1pt);

\draw[black, thick](0,-1)circle(2pt);
\filldraw[white, thick](0,-1) circle (1pt);

\draw[black, thick](-2,1)circle(2pt);
\filldraw[white, thick](-2,1) circle (1pt);

\filldraw[fill=red!10!white, draw=black,dashed] (3,1)--(5,1)--(5,-1)--(3,1);

\draw (-6,-0.5) node{$S^c_{0}$};
\draw (-1.3,-0.3) node{$S^c_{1}$};
\draw (4.3,0.3) node{$S^c_{2}$};

    \end{tikzpicture}
\]

\[
    		{\begin{tikzpicture}
            \node (v0) at (-10,1.5) {$0$};
            \node (v1) at (-8,1.5) {$\mathcal{O}(-2)$};
            \node (v2) at (-4.5,1.5) {$\mathcal{O}(-1)^2$};
            \node (v3) at (-1,1.5) {$\mathcal{O}$};
            \node (v4) at (0.5,1.5) {$\mathcal{O}_p$};
            \node (v5) at (2,1.5) {$0$};

            \draw[->]  (v0) edge 
            (v1);            
             \draw[->]  (v1) edge node[above=2pt]{$\begin{pmatrix} x_1 - x_2 \\ x_1 -x_0 \end{pmatrix}$}
            (v2);
            \draw[->]  (v2) edge node[above=2pt]{$\begin{pmatrix} x_0 - x_1 & x_1 - x_2 \end{pmatrix}$} (v3);
            \draw[->]  (v3) edge (v4);
            \draw[->]  (v4) edge
            (v5);
            
  \end{tikzpicture} }
	\]
Alternatively, the Borel-Moore homology of $S^c_1$ is the relative homology of the closure of the cyan triangle to its three verticies. For $S^c_2$ its the relative homology of the closure of the red triangle to its boundary, which is homeomorphic to $S^{1}$. 
\end{example}

\begin{example}
Consider the structure sheaf of $X_{\Sigma_1}$ which arises as a pushforward by the identity map. For any cube stratified torus, $S^{c}_{[0]} = \{ 0 \}$ (the structure sheaf is in the collection and its stratum is the point 0), so that we expect the minimal resolution to be trivial. Indeed, $V = \{ 0 \} \in \T^n$. By \Cref{theorem: min res of subvarieties}, the only non-trivial intersection is the point $\{ 0 \}$, which gives the trivial minimal resolution.  
\end{example}

\begin{example}
Consider the map $f: \P^1 \rightarrow \P^2$ given by the morphism of lattices $f: (a) \rightarrow (2a,a)$. The sheaf $u_{*}\cO_{\P^1}$ is the quadric $x^2_0 -x_1x_2$. We have an induced dual map $f^\vee: (a,b) \rightarrow 2a + b$, so that $V$ is a one dimensional subtorus in $M_{2,\R}$ defined by the equation $2x + y = 0$.

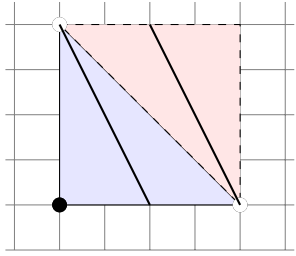
\[
\begin{tikzpicture}[scale=1.2]
\draw[step=0.5cm,gray,very thin](-2.6,-1.5) grid (0.6,1.25);

\filldraw[fill=blue!10!white, draw=black] (-2,-1) -- (0,-1) -- (-2,1)--(-2,-1);

\filldraw[fill=red!10!white, draw=black,dashed] (-2,1)--(0,1)--(0,-1)--(-2,1);

\draw[black, thick](-2,-1)circle(2pt);
\filldraw[black, thick](-2,-1) circle (2pt);

\draw[black, thick](0,-1)circle(2pt);
\filldraw[white, thick](0,-1) circle (2pt);

\draw[black, thick](-2,1)circle(2pt);
\filldraw[white, thick](-2,1) circle (2pt);

\draw[black,thick] (-2, 1) -- (-1,-1);
\draw[black,thick] (-1, 1) -- (0,-1);  
    \end{tikzpicture}
\]
The subtorus intersects each stratum once: the intersection with the cyan stratum is given by a half-open interval, which has no compactly supported homology; the intersection with the red stratum is given by an open interval which has a one-dimensional compactly supported homology in degree 1, and the zero-dimensional stratum has a trivial intersection, which has a one-dimensional compactly supported homology in degree 0. By \Cref{theorem: min res of subvarieties}, the quadric has a minimal resolution with the line bundle $\mathcal{O}(-2)$ in degree 1 and $\mathcal{O}$ in degree 0. One checks it is simply:
\[
\begin{tikzpicture}
            \node (v0) at (-9,1.5) {$ \mathcal{O}(-2)$};
            \node (v1) at (-5,1.5) {$\mathcal{O}$};
            \draw[->]  (v0) edge node[above=2pt]{$\begin{pmatrix}  x^{2}_2-x_0x_1 \end{pmatrix}$}
            (v1);
\end{tikzpicture}
\]

\end{example}

\begin{example}\label{cubic}
Consider the map $f: \P^1 \rightarrow \P^2$ given by the morphism of lattices $f: (a) \rightarrow (2a,3a)$. The sheaf $u_{*}\cO_{\P^1}$ is the normalization of the cubic $x^3_1 -x^2_0x_2$. We have an induced dual map $f^\vee: (a,b) \rightarrow 2a + 3b$, producing a one dimensional subtorus in $M_{2,\R}$ defined by the equation $2x + 3y = 0$. 
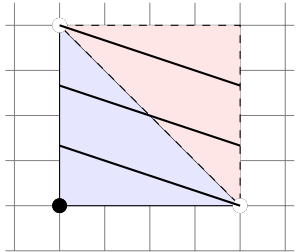
\[
\begin{tikzpicture}[scale=1.2]
\draw[step=0.5cm,gray,very thin](-2.6,-1.5) grid (0.6,1.25);

\filldraw[fill=blue!10!white, draw=black] (-2,-1) -- (0,-1) -- (-2,1)--(-2,-1);

\filldraw[fill=red!10!white, draw=black,dashed] (-2,1)--(0,1)--(0,-1)--(-2,1);

\draw[black, thick](-2,-1)circle(2pt);
\filldraw[black, thick](-2,-1) circle (2pt);

\draw[black, thick](0,-1)circle(2pt);
\filldraw[white, thick](0,-1) circle (2pt);

\draw[black, thick](-2,1)circle(2pt);
\filldraw[white, thick](-2,1) circle (2pt);

\draw[black,thick] (-2, 1) -- (0,0.332);
\draw[black,thick] (-2, 0.332) -- (0,-0.334);  
\draw[black,thick] (-2, -0.334) -- (0,-1);
    \end{tikzpicture}
\]
Computing the compactly supported cohomology of the intersections with the strata, we get the following complex:
\[
\begin{tikzpicture}\label{min resolution of the cubic}
            \node (v0) at (-7,1.5) {$ \mathcal{O}(-2)^{2}$};
            \node (v1) at (-3,1.5) {$\mathcal{O}(-1)\oplus \mathcal{O}$};
            \draw[->]  (v0) edge node[above]{$\begin{pmatrix} x_1 &  x_0  \\ -x_0x_2 & -x^2_1 \end{pmatrix}$}
            (v1);
\end{tikzpicture}
\]
While the compactly supported cohomology told us the terms of the complex, at this point we don't know how to determine the differential. Using Macaulay2, the normalization of the push-forward of the  the cubic is given by the following ring:
\[ N := \C[x_0,x_1,x_2,t]/(tx_1-x_0x_2,tx_0 - x^2_1, t^2-x_1x_2) \] viewed as a $\C[x_0,x_1,x_2]$-module via the trivial inclusion of the semigroup rings. One checks that the resolutions is: 
\[
\begin{tikzpicture}
            \node (v0) at (-7,1.5) {$ \mathcal{O}(-2)^{2}$};
            \node (v1) at (-3,1.5) {$\mathcal{O}(-1)\oplus \mathcal{O}$};
            \node (v2) at (0,1.5) {$N$};
            \draw[->]  (v0) edge node[above]{$\begin{pmatrix} x_1 &  x_0  \\ -x_0x_2 & -x^2_1 \end{pmatrix}$}
            (v1);
            \draw[->]  (v1) edge node[above]{$\begin{pmatrix} t &  1   \end{pmatrix}$}
            (v2);
\end{tikzpicture}
\]

\end{example}

\begin{example}\label{rational map}
Let $f$ be an arbitrary map of lattices, $f: N_1 \to N_2$. Consider the graph of $f$, $\Gamma(f): N_1 \to N_1 \times N_2$. Define a fan $\Sigma_3$ in $N_{1,\R}$ by pulling back the product of fans $\Sigma_1$ and $\Sigma_2$ along $\Gamma(f)_{\R}$. Then, $\Gamma(f)$ becomes a fan-preserving map for $\Sigma_3$, so that we obtain a minimal projective resolution of $\Gamma(u)_{*}\mathcal{O}_{\Sigma_3}$. If $f$ is the identity map and $\Sigma_1 = \Sigma_2$ this gives the minimal resolution of the diagonal. As it is known, the resolution of the diagonal translates to a resolution of any coherent sheaf on $X$. More generally, the resolution of $\Gamma(u)_{*}\mathcal{O}_{\Sigma_3}$ translates to a resolution of $u_{*} \circ \phi^*(F)$ for any coherent sheaf $F$ on $X_{\Sigma_1}$, where $\phi$ is the morphism associated to the identity morphism of lattices. When $f$ is the identity map, both $f$ and $\phi$ are birational morphisms. This shows that our construction allows to produce a large class of resolutions. 
\end{example}

\begin{example}
The Frobenious map $F^k: N \to N, a \to ka$  is injective and fan-preserving. The zero fiber $V$ is a collection of points $(\{0, 1/k, 2/k, ..., k-1/k\})^n$ (by viewing torus as a quotient space of the unit cube). Since the compactly supported homology of a point is concentrated in degree 0 and is 1-dimensional, by \Cref{theorem: min res of subvarieties}, $F^k_{*}\mathcal{O}$ is isomorphic to a direct sum of line bundles in the cube stratification collection with various multiplicities. This recovers the result of Thompsen \cite{Th}. The multiplicities of the summands are counted by the number of points of $X$ in the corresponding strata. For topological reasons, for high enough $k$, the set $V$ must intersect each stratum. Therefore, for high enough $k$, the summands of the push-forward of the structure sheaf under the Frobenious must generate the derived category. For further reference, see \cite{HHL} and \cite{Rouquier}.
\end{example}

\subsection{Explicit Minimal Resolutions of Pushforward Sheaves}\label{subsec: explicit minimal resolutions}

Although \Cref{theorem: min res of subvarieties} shows the existence of a minimal projective resolution of $u_{*}\mathcal{O}_{X_{\Sigma_1}}$by an explicit collection of line bundles and describes how to compute its terms, it does not provide any insight on how to explicitly construct such a resolution. The goal of section is to explain how to amend this. 

First let us recall how we proved the existence of the minimal resolution of $u_{*}\mathcal{O}_{X_{\Sigma_1}}.$ We showed that there is an object $M_{\Sigma}:= S^{-1} \circ F ( \rho^{R}_{\Lambda_c}(\overline{v}^{!}j_{\{0\}*}\C_{\{0\}})) \in D(\RMod A^{c}_{Ent})$, quasi-isomorphic to a pure module in degree 0, and whose image under the functor $- \otimes_{A} T$ gives $u_{*}\mathcal{O}_{X_{\Sigma_1}}$. The goal of this section is to show that the minimal projective resolution of $M_{\Sigma}$ can be obtained from a projective resolution of the constant representation of the exit paths quiver of $V$ in $\RMod A^{V}_{Exit}$. Then, we can refer to the literature, where the explicit constructions of minimal (and non-minimal) projective resolutions of quiver representations with relations are well studied. For example, the paper of Greene and Solberg \cite{MPR} gives a computational way of constructing an explicit minimal projective resolution for any finite-dimensional path algebra with relations. 

\begin{lemma}
There exists an isomorphism:
\[
h_{*}\C_{V}[dim \ V] \cong \overline{v}^{!}j_{\{0\}*}\C_{\{0\}}
\]
\end{lemma}
\begin{proof}
Consider the following Cartesian square, where $\overline{v}$ is a proper map. 

\[
\begin{tikzcd}
V = \overline{v}^{-1}(0) \ar[r, "h"] \ar[d, "p"]& \T^n \ar[d, "\overline{v}"] \\ {0} \ar[r, "j_{\{0\}}"] & \T^m
\end{tikzcd}
\]
Then 
\begin{align*}
 \overline{v}^{!}j_{\{0\}*}\C_{\{0\}} & \cong h_{*} \circ p^{!}(\C_{\{0\}}) & \text{by proper base change} \\
& \cong h_{*} (\C_{V}[dim V]) & \text{since $V$ is a smooth manifold} \\
\end{align*}    
\end{proof}

\begin{lemma}
Let $P^{\bullet}_{K}$ be the minimal projective resolution of the constant representation $K$ in $\RMod A^{V}_{Exit}$. Then, $\mathcal{I}^{\bullet}:= F_{V}^{-1} \circ \D(P^{\bullet}_{K})$ is the minimal injective resolution of $\C_{V}[dim V]$.     
\end{lemma}

\begin{proof}
\begin{align*}
F_{V}^{-1} \circ \D(K) & \cong F_{V}^{-1}(K) & \text{$\D$ preserves the constant representation}\\
& \cong \C_{V} & \text{$F$ takes constant sheaves to constant representations by Yoneda}
\end{align*}

Since $\D$ takes projectives to injectives and the exodromy functor is an exact equivalence, the result follows. 
\end{proof}

Now that we have produced an explicit, minimal, injective resolution for $\C_{V}[dim V]$, we are left to justify that $\rho^{R}_{\Lambda_c}\circ h_{*}(\mathcal{I}^{\bullet})$ is a computable explicit injective complex of sheaves. Applying $S^{-1} \circ F$ to the resulting complex yields an explicit projective resolution for $M_{\Sigma}:= S^{-1} \circ F ( \rho^{R}_{\Lambda_c}(\overline{v}^{!}j_{\{0\}*}\C_{\{0\}}))$.

\begin{proposition}\label{prop: explicit minimal resolution}
Let $\mathcal{I}^{\bullet}$ be the minimal injective resolution of $\C_{V}[dim V]$ as above. Then,
\[
\rho^{R}_{\Lambda_c}(\overline{v}^{!}j_{\{0\}*}\C_{\{0\}})) \cong  \rho^{R}_{\Lambda_c}\circ h_{*}(\mathcal{I}^{\bullet})
\] 
is an explicit complex of injective co-probe sheaves.      
\end{proposition}

\begin{proof}
We need to compute the image of $\mathcal{I}_{[a]}$ under the composition $\rho^{R}_{\Lambda_{c}} \circ h_{*}$. We have:

\begin{align*}
 \Hom_{D(Sh_{S_c}(\T^n))}(F,\rho^{R}_{\Lambda_{c}} \circ h_{*}(\mathcal{I}_{[a]})) & \cong \Hom_{D^{b}_{c}(\T^n))}(F,h_{*}(\mathcal{I}_{[a]})) & \text{by adjunction} \\
& \cong \Hom_{D^{b}_{c}(\T^n)}(h^{*}F,\mathcal{I}_{[a]}) & \text{by adjunction} \\
& \cong \Hom_{D(Sh_{S^{V}_c}(V))}(h^{*}F,\mathcal{I}_{[a]}) & \text{by definition of $S^{V}_{c}$} \\
& \cong ((h^{*}F)_{[a]})^{*} \cong (F_{[a]})^{*} &\text{since $\mathcal I_{[a]}$ is the co-probe sheaf} \\
\end{align*}
Since co-representing the stalk functor determines the object up to a unique isomorphism, we obtain: 
\[
\rho^{R}_{\Lambda_{c}} \circ h_{*}(\mathcal{I}^{V}_{[a]}) \cong \mathcal{I}_{[a]} \in Sh_{S_{c}}(\T^n)
\]

Hence, we can compute $\rho^{R}_{\Lambda_c}\circ h_{*}(\mathcal{I}^{\bullet})$ by applying the functor termwise. 

\end{proof}

\begin{example}
In this example, we perform the steps above to compute the explicit minimal resolution of the cubic from \Cref{cubic}.

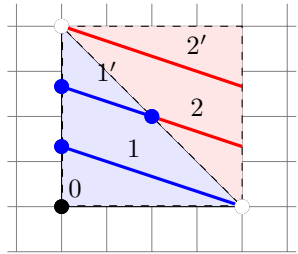
\[\begin{tikzpicture}[scale=1.2]
\draw[step=0.5cm,gray,very thin](-2.6,-1.5) grid (0.6,1.25);

\filldraw[fill=blue!10!white, draw=black] (-2,-1) -- (0,-1) -- (-2,1)--(-2,-1);

\filldraw[fill=red!10!white, draw=black,dashed] (-2,1)--(0,1)--(0,-1)--(-2,1);

\draw[black,dashed] (-2,-0.99) -- (0,-0.99);
\draw[black,dashed] (-1.99,1) -- (-1.99,-1);

\draw[black, thick](-2,-1)circle(2pt);
\filldraw[black, thick](-2,-1) circle (2pt);

\draw[black, thick](0,-1)circle(2pt);
\filldraw[white, thick](0,-1) circle (2pt);

\draw[black, thick](-2,1)circle(2pt);
\filldraw[white, thick](-2,1) circle (2pt);

\draw[red, very thick] (-1.93, 0.98) -- (0,0.332);
\draw[blue, very thick] (-2, 0.332) -- (-1,0);  
\draw[blue, very thick] (-2, -0.334) -- (-0.08,-0.97); 
\draw[red, very thick] (-1, 0) -- (0,-0.334);

\draw[blue, thick](-2, 0.332)circle(2pt);
\filldraw[blue, thick](-2, 0.332) circle (2pt);

\draw[blue, thick](-2, -0.334)circle(2pt);
\filldraw[blue, thick](-2, -0.334) circle (2pt);
\draw[blue, thick](-1, 0)circle(2pt);
\filldraw[blue, thick](-1, 0) circle (2pt);

\draw (-1.2,-0.35) node{$1$};
\draw (-1.5,0.5) node{$1'$};
\draw (-1.85,-0.8) node{$0$};
\draw (-0.5,0.1) node{$2$};
\draw (-0.5,0.8) node{$2'$};

\end{tikzpicture}
\]
The corresponding exit paths algebra $A^{V}_{cExit}$ is given below: 
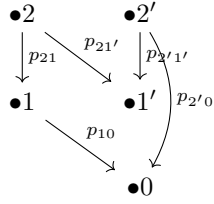
\[\begin{tikzcd}
	{} & \bullet2 & {\bullet 2'} \\
	& {\bullet 1} & {\bullet 1'} \\
	&& {\bullet 0}
	\arrow["p_{21}",from=1-2, to=2-2]
	\arrow["p_{21'}", from=1-2, to=2-3]
	\arrow["p_{2'1'}", from=1-3, to=2-3]
	\arrow["p_{2'0}", bend left, from=1-3, to=3-3]
	\arrow["p_{10}",from=2-2, to=3-3]
\end{tikzcd}\]

Running the algorithm in \cite{MPR} on the constant representation, we obtain that the minimal projective resolution  in $\RMod A^{V}_{cExit}$ is

\[\begin{tikzpicture}

          \node (v0) at (-9,1.5) {$ P_{2'} \oplus P_{2}$};
          \node (v1) at (-3,1.5) {$P_{1'}\oplus P_{0}$};
          \node(v2) at (0,1.5) {$K$};

            \draw[->]  (v0) edge node[above]{$\begin{pmatrix} p_{2'1'} &  p_{21'}  \\ -p_{2'0} & -p_{10} p_{21} \end{pmatrix}$}
            (v1);
            \draw[->] (v1) edge node[above]{$\begin{pmatrix} 1 &  1\end{pmatrix}$}(v2);
\end{tikzpicture}
\]
Applying $S \circ F \circ \rho^{R}_{\Lambda_c} \circ h_{*}\circ F^{-1}_{V} \circ \D$ to this complex yields a resolution of $M_{\Sigma}$ in $\RMod A^{c}_{Ent}$:

\[\begin{tikzpicture}

          \node (v0) at (-7,1.5) {$ P_{2} \oplus P_{2}$};
          \node (v1) at (-3,1.5) {$P_{1}\oplus P_{0}$};

            \draw[->]  (v0) edge node[above]{$\begin{pmatrix} c' &  a'  \\ -b a' & -cc' \end{pmatrix}$}
            (v1);
\end{tikzpicture}
\]

\[
\begin{tikzpicture}[scale=1.5]

\draw[step=0.5cm,gray,very thin](-2.6,-1.5) grid (0.6,1.25);

\filldraw[fill=blue!10!white, draw=black] (-2,-1) -- (0,-1) -- (-2,1)--(-2,-1);

\filldraw[fill=red!10!white, draw=black,dashed] (-2,1)--(0,1)--(0,-1)--(-2,1);

\draw[black, thick](-2,-1)circle(2pt);
\filldraw[black, thick](-2,-1) circle (2pt);

\draw[black, thick](0,-1)circle(2pt);
\filldraw[white, thick](0,-1) circle (2pt);

\draw[black, thick](-2,1)circle(2pt);
\filldraw[white, thick](-2,1) circle (2pt);

\draw (-2.2,-1.2) node{$0$};
\draw (-1.3,-0.1) node{$1$};
\draw (-0.5,0.7) node{$2$};

\draw[black, thick](-1.4,-0.4)circle(1.5pt);
\filldraw[black, thick](-1.4,-0.4) circle (1.5pt);

\draw[black, thick](-0.6,0.5)circle(1.5pt);
\filldraw[black, thick](-0.6,0.5) circle (1.5pt);

\draw[->, orange, thick] (-2,-1) -- (-1.44,-0.44);

\draw[->,orange,thick] (0,-1) -- (-1.33,-0.4);
\draw[->,orange,thick] (-2,1) -- (-1.38,-0.34);
\draw[->, violet, thick] (-1.4,-0.4) -- (-2,0);
\draw[->, violet, thick] (0,0) -- (-0.57,0.5);
\draw[->, violet, thick] (-1.4,-0.4) -- (-1,-1);
\draw[->, violet, thick] (-1,1) -- (-0.62,0.52);
\draw[->, violet, thick] (-1.4,-0.4) -- (-0.62,0.46); 

\draw (-1.7,-0.55) node{$a$};
\draw (-1.5,0.2) node{$b$};
\draw (-0.85,-0.5) node{$c$};

\draw (-1,0.25) node{$a'$};
\draw (-0.15,0.37) node{$c'$};
\draw (-0.67,0.82) node{$b'$};

    \end{tikzpicture}
\]

Applying $-\otimes_{A}{T}$ to this complex, we obtain the projective resolution we constructed in \eqref{min resolution of the cubic}.

\[
\begin{tikzpicture}
            \node (v0) at (-7,1.5) {$ \mathcal{O}(-2)^{2}$};
            \node (v1) at (-3,1.5) {$\mathcal{O}(-1)\oplus \mathcal{O}$};
            \node (v2) at (0,1.5) {$N$};
            \draw[->]  (v0) edge node[above]{$\begin{pmatrix} x_1 &  x_0  \\ -x_0x_2 & -x^2_1 \end{pmatrix}$}
            (v1);
            \draw[->]  (v1) edge node[above]{$\begin{pmatrix} t &  1   \end{pmatrix}$}
            (v2);
\end{tikzpicture}
\]
\end{example}

\subsection{Explicit CW Resolutions of Pushforward Sheaves}

If one is unhappy with the need to run an algorithm to obtain an explicit answer, we can sacrifice minimality for the sake of having a more universal description of the differentials of the complex. This complex recovers the one obtained by Hanlon, Hicks and Lazarev \cite{HHL}. 

Recall that the starting point in producing an explicit minimal resolution of $M_{\Sigma}$ was constructing a minimal projective resolution of the constant representation in the exit paths algebra of the fiber $V$, which after applying $F^{-1}_{V} \circ \D$, equivalently, is a minimal injective resolution of the constant sheaf shifted by the dimension of $V$. In this section, we show how to produce another injective resolution of the constant sheaf on $V$, which comes from the CW stratification of $V$. This resolution is not minimal; however, it is cellular and hence can be described by a closed form differential.

We first present an explicit injective resolution of the constant sheaf in $Sh_{S^V_{cw}}(V)$ with the summands of the $i^{th}$ term indexed by the $i$-cells.  To begin, pick a basis for the vector space $v^{-1}(0) \subset \R^n$. This endows every top dimensional cell of $\widetilde{V}$ with an orientation. Since $\widetilde{V}$ is regular, the degree of every attaching map is $\pm1$.  Every exit path in $V$ lifts to a unique attaching map $\partial$ in the universal cover, allowing one to construct the following complex.

\begin{proposition}\label{CW Injective resolution}
Consider the injective sheaf
\[
\mathcal{F}_t   := \bigoplus_{e_i \in T_{V} \ dim(e_i)  = t } \mathcal{I}_{e_{i}}
\]
and define maps
\begin{align*}
d_t:  \mathcal F_{t} & \to \mathcal F_{t-1}\\
 \mathcal{I}_{e_{i}}  & \xmapsto{ \sum(-1)^{\text{deg}(\widetilde{\partial_{ij}})} {\partial_{ij}}}  \mathcal{I}_{e_{j}} 
    \end{align*}
    and
\begin{align*}
d_{k+1}: \C_V & \to  \mathcal F_0\\
 \C_V  & \xto{\partial{e_i}} \mathcal{I}_{e_{i}}
\end{align*}
where $\partial$ is the canonical restriction maps. 
Then,  
\begin{equation}
0 \rightarrow \C_{V}  \xto{d_{k+1}} \mathcal{F}_{k}   \xto{d_{k}} \mathcal{F}_{k-1} \xto{d_{k-1}} \cdots \mathcal{F}_{1} \xto{d_{1}} \mathcal{F}_{0} \xto{} 0 \label{eq: injective complex}
\end{equation}
is an acyclic complex.   
\end{proposition}

\begin{proof}
First, we verify that $d_{k} \circ d_{k+1} = 0$.  We check on a summand of $\F_{k-1}$:
\begin{align*}
d_{k-1} \circ {d_k}_{\big|\mathcal{I}_{e_{j}}}
& = \sum_{e_{i}, e_{j} \leq e_{i}}\sum(-1)^{\text{deg}(\widetilde{\partial_{ij}})} {\partial_{ij}}\circ\partial_{e_i}
\end{align*}

We first show that for every $k-1$ dimensional cell in $\widetilde{V}$ there are exactly two $k$-dimensional cells whose closure contains it. Since the cells are homeomorphic to the intersections of $\widetilde{V}$ with the cells of $\R^n$ that are locally given by linear subspaces, each $k-1$-dimensional cell is locally given by a hyperplane in $\widetilde{V}$. Take a $k-1$-dimensional cell $\widetilde{e_i}$. Its associated hyperplane divides $\widetilde{V}$ into 2 open sets. Pick an arbitrary cell $\widetilde{e_j}$ whose closure contains $\widetilde{e_i}$. Since $\widetilde{e_i}$ is $k-1$-dimensional, $\widetilde{e_j}$ must be $k$-dimensional. It must contain a point in either of the two halves divided by the hyperplane of $\widetilde{e_i}$ because otherwise it would not be $k$-dimensional. Since the strata are open and convex, $\widetilde{e_j}$ lies entirely in one of the two halves. Since the stratification satisfies the axiom of frontier, the closure of $\widetilde{e_j}$ contains $\widetilde{e_i}$. Suppose there is another open cell $\widetilde{e_j'}$ whose closure contains $\widetilde{e_j}$. Then it must lie entirely in one of the two halves, but if it lies in the same half as $\widetilde{e_j}$, we get a contradiction on disjointness of the cells since from the same point in the boundary, we get two half-balls which non-empty intersection is contained in both cells.
 
Hence, the restricted differential is just a signed sum of two restriction maps. Because every top cell is uniformly oriented by the chosen basis of $\widetilde V$, the degree of the two boundary maps have the opposite sign, which means that the sum is 0. Hence, $d_{k} \circ d_{k+1} = 0$. We show that the remaining differentials compose to 0 and that the complex is exact by checking on stalks.  

Consider $x \in e_j \subset V$.   Then 
\[
(\mathcal{I}_{e_i})_x = \begin{cases}
 \C  & \text{if } e_j \subset \overline{e_i} \\
 0 & \text{otherwise}
\end{cases}
\]
The sequence of maps 
\begin{equation} \label{eq: stalk at x}
(\cF_k)_x \to ... \to (\cF_0)_x 
\end{equation} is precisely the cellular relative homology complex of the entrance space of a lift $\widetilde x \in \widetilde V$ relative to its closure. 
In particular, it is a complex (at $x$ for all $x$).  Furthermore, since the entrance spaces are homeomorphic to an open $k$-dimensional ball and its closure is a closed $k$-ball, the complex \eqref{eq: stalk at x} only has top homology.  This means that \eqref{eq: injective complex} is exact at $x$ for all $x$, as desired.
\end{proof}

We now show that the explicit minimal injective resolution of the constant sheaf in $S^{V}_{cw}$ produces an explicit (non-minimal) injective resolution of the constant sheaf in $S^{V}_{c}(V)$. Then, we can proceed by applying a sequence of functors as in \Cref{subsec: explicit minimal resolutions}, to obtain a non-minimal resolution of $u_{*}\mathcal{O}_{X_{\Sigma_1}}$by line bundles. 

\begin{proposition}
Consider the following pair of adjoint functors: 
\[
\begin{tikzcd}
D(Sh_{S^{V}_{cw}}(V))\ar[r, bend right, "\rho^{R}_{\Lambda_c} \circ i_{*}"'] & D(Sh_{S^{V}_c}(V)) \ar[l, bend right, "i^{*}"']
\end{tikzcd}
\]
where $i$ is the tautological inclusion. Given an injective resolution of the constant sheaf $\mathcal{I}^{\bullet} \in Sh_{S^{V}_{cw}}(V)$, $\rho^{R}_{\Lambda_c} \circ i_{*}(\mathcal{I}^{\bullet})$ is an injective resolution of the constant sheaf in $Sh_{S^{V}_c}(V)$, computed using the formula:
\[
\rho^{R}_{\Lambda_c} \circ i_{*}(\mathcal I_{e_{i}}) \cong \mathcal I_{\phi(e_{i})}.
\]
\end{proposition}

\begin{proof}
We have:
\begin{align*}
& \C_{V} \cong (\rho^{R}_{\Lambda_c} \circ i_{*}) \circ i^{*} (\C_{V}) & \text{$i^{*}$ is fully faithful} \\ 
& \cong \rho^{R}_{\Lambda_c} \circ i_{*} (\C_{V}) \\
\end{align*}
Since the right wrapping takes co-probes to co-probes, the claim follows. 

\end{proof}

\begin{example} In this example, we construct an explicit resolution of the torus-fixed point in $\P^2$. The choice of the differential and the terms is coming from the cellular resolution of the constant sheaf on $V = \T^n$.
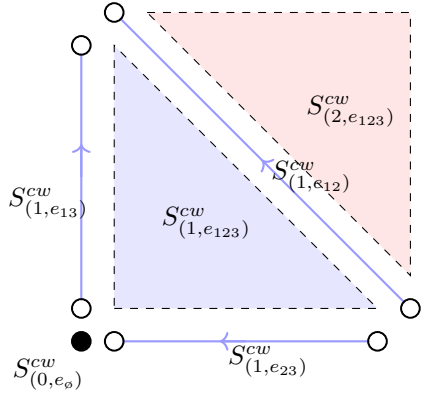
\[
    \begin{tikzpicture}[scale=1.75]

\filldraw[black, thick](-2.25,-1.25)circle(2pt);

\filldraw[fill=blue!10!white, draw=black,dashed] (-2,-1) -- (0,-1) -- (-2,1)--(-2,-1);

\draw[blue!40!white,thick] (-2.25,-1) -- (-2.25,1);
\draw[->,blue!40!white,thick] (-2.25,0) -- (-2.25,0.25);

\draw[black, thick](-2.25,-1)circle(2pt);
\filldraw[white, thick](-2.25,-1) circle (1.5pt);

\draw[black, thick](-2.25,1)circle(2pt);
\filldraw[white, thick](-2.25,1) circle (1.5pt);

\draw[blue!40!white,thick] (0.25,-1) -- (-2,1.25);
\draw[->,blue!40!white,thick] (-0.875,0.125) -- (-0.878,0.128);

\draw[black, thick](0.25,-1)circle(2pt);
\filldraw[white, thick](0.25,-1) circle (1.5pt);

\draw[black, thick](-2,1.25)circle(2pt);
\filldraw[white, thick](-2,1.25) circle (1.5pt);

\draw[blue!40!white,thick] (0,-1.25) -- (-2,-1.25);
\draw[->,blue!40!white,thick] (-1,-1.25) -- (-1.2,-1.25);

\draw[black, thick](0,-1.25)circle(2pt);
\filldraw[white, thick](0,-1.25) circle (1.5pt);

\draw[black, thick](-2,-1.25)circle(2pt);
\filldraw[white, thick](-2,-1.25) circle (1.5pt);

\filldraw[fill=red!10!white, draw=black,dashed] (-1.75,1.25)--(0.25,1.25)--(0.25,-0.75)--(-1.75,1.25);

\draw (-2.5,-1.5) node{$S^{cw}_{(0,e_{\o})}$};
\draw (-1.3,-0.35) node{$S^{cw}_{(1,e_{123})}$};
\draw (-2.5,-0.2) node{$S^{cw}_{(1,e_{13})}$};
\draw (-0.5,0) node{$S^{cw}_{(1,e_{12})}$};
\draw (-0.83,-1.4) node{$S^{cw}_{(1,e_{23})}$};

\draw (-0.2,0.5) node{$S^{cw}_{(2,e_{123})}$};

    \end{tikzpicture}
\]

\begin{enumerate}
    \item The universal cover of $\T^2$, $\R^2$, is assigned a clockwise orientation, orienting the two top-dimensional strata clockwise. The other strata are oriented arbitrarily.  
    \item We choose the lifts for the top dimensional strata so that they lie in the same fundamental domain. We choose $((0,0,-1), e_{123})$ to be the lift for $(1,e_{123})$ and $((0,0,-2), e_{123})$ to be the lift for $(2,e_{123})$.
    \item We now consider all exit paths from the two dimensional strata to the one dimensional strata. Take the exit path from $S^{cw}_{(1,e_{123})}$ to  $S^{cw}_{(1,e_{12})}$. This exit path lifts to the attaching map  $e_{(0,0,-1),{12}} \subset \overline{e_{(0,0,-1),{123}}}$. Since $(0,0,-1) - (0,0,-1) = 0$, the corresponding morphism of line bundles is the identity morpshism. Since the induced orientation on the boundary is different from the chosen orientation, the degree of the attaching map is $-1$. The other exit paths for the blue triangle also give the identity morpshisms but with the degree $1$ since the orientation agrees. Consider the exit path from $S^{cw}_{(2,e_{123})}$ to $S^{cw}_{(1,e_{23})}$. We get a unique lift for $(1,e_{23})$, $((1,0,-2), e_{23})$ and the corresponding attaching map. We have $(1,0,-2) - (0,0,-2) = (1,0,0) = x_0$ giving us a morphism of line bundles $\mathcal{O}(-2) \xrightarrow{-x_0}\mathcal{O}(-1)$. For $S^{cw}_{(1,e_{12})}$, the lift is $((0,0,-1),e_{12})$, so that the morphism is given by the monomial $x_2$. For $S^{cw}_{(1,e_{13})}$, the lift is $((0,1,-2),e_{12})$, so that the morphism is given by $x_1$.
    \item We have constructed the first two terms of the resolution:

    \begin{tikzpicture}
            \node (v0) at (-10,1.5) {$\mathcal{O}(-1) \oplus \mathcal{O}(-2)$};
            \node (v1) at (-5,1.5) {$\mathcal{O}(-1)^3$};
            \draw[->]  (v0) edge node[above]{$\begin{pmatrix} - 1 & x_2  \\ 1 & -x_0 \\ 1 & -x_1 \end{pmatrix}$}
            (v1);
      \end{tikzpicture}
   \item We now study the exit paths from the one dimensional strata to $S^{cw}_{(0,e_{\o})}$. Consider the positively signed exit path from $S^{cw}_{(1,e_{12})}$ to $S^{cw}_{(0,e_{\o})}$, which lifts $(0,e_{\o})$ to $((1,0,-1),e_{\o})$. We get a monomial $(1,0,-1) - (0,0,-1) = (1,0,0) = x_0$. For the other exit path, the lift for $(0,e_{\o})$ is $((0,1,-1),e_{\o})$, giving the monomial $-x_1$. Repeating the same procedure for the remaining strata, we compete our resolution:
   \begin{tikzpicture}
            \node (v0) at (-10,1.5) {$\mathcal{O}(-1) \oplus \mathcal{O}(-2)$};
            \node (v1) at (-5,1.5) {$\mathcal{O}(-1)^3$};
            \node(v2) at (1.5,1.5)
            {$\mathcal{O}$};
            
            \draw[->]  (v0) edge node[above]{$\begin{pmatrix} - 1 & x_2  \\ 1 & -x_0 \\ 1 & -x_1 \end{pmatrix}$}
            (v1);
            \draw[->]  (v1) edge node[above]{$\begin{pmatrix} x_0 - x_1 & x_2 - x_1   & x_0 - x_2 \end{pmatrix}$}
            (v2);
            
      \end{tikzpicture}

\end{enumerate}
\end{example}

\begin{example}
Consider the cubic from \Cref{cubic}. The induced CW stratification of $\T^1$ is depicted below.
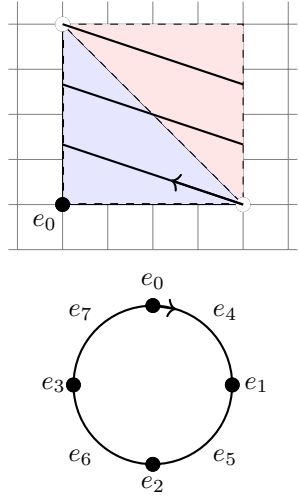
\[
\begin{tikzpicture}[scale=1.2]
\draw[step=0.5cm,gray,very thin](-2.6,-1.5) grid (0.6,1.25);

\filldraw[fill=blue!10!white, draw=black] (-2,-1) -- (0,-1) -- (-2,1)--(-2,-1);

\filldraw[fill=red!10!white, draw=black,dashed] (-2,1)--(0,1)--(0,-1)--(-2,1);

\draw[black,dashed] (-2,-0.99) -- (0,-0.99);
\draw[black,dashed] (-1.99,1) -- (-1.99,-1);

\draw[black, thick](-2,-1)circle(2pt);
\filldraw[black, thick](-2,-1) circle (2pt);

\draw[black, thick](0,-1)circle(2pt);
\filldraw[white, thick](0,-1) circle (2pt);

\draw[black, thick](-2,1)circle(2pt);
\filldraw[white, thick](-2,1) circle (2pt);

\draw[black,thick] (-2, 1) -- (0,0.332);
\draw[black,thick] (-2, 0.332) -- (0,-0.334);  
\draw[black,thick] (-2, -0.334) -- (0,-1);

\draw[black, thick](-1,-3)circle(25pt);

\draw[black, thick](-1,-3.88)circle(2pt);
\filldraw[black, thick](-1,-3.88) circle (2pt);
\draw[black, thick](-1,-2.12)circle(2pt);
\filldraw[black, thick](-1,-2.12) circle (2pt);
\draw[black, thick](-1.88,-3)circle(2pt);
\filldraw[black, thick](-1.88,-3) circle (2pt);
\draw[black, thick](-0.12,-3)circle(2pt);
\filldraw[black, thick](-0.12,-3) circle (2pt);

\draw (-1,-1.85) node{$e_0$};
\draw (-1,-4.1) node{$e_2$};
\draw (0.15,-3) node{$e_1$};
\draw (-2.1,-3) node{$e_3$};

\draw (-0.2,-2.2) node{$e_4$};
\draw (-0.2,-3.8) node{$e_5$};
\draw (-1.8,-3.8) node{$e_6$};
\draw (-1.8,-2.2) node{$e_7$};

\draw (-2.2,-1.2) node{$e_0$};
\draw[->,black,thick] (0,-1) -- (-0.8,-0.73);
\draw[->,black,thick] (-1,-2.12) -- (-0.74,-2.16);

    \end{tikzpicture}
\]
Putting a clockwise orientation, we can construct an injective resolution of the constant sheaf on $V = \T^1$:

\[
 \begin{tikzpicture}
            \node (v0) at (-10,1.5) {$\C_{\T^1}$};
            \node (v1) at (-6,1.5) {$\mathcal{I}_{e_4}\oplus\mathcal{I}_{e_5}\oplus\mathcal{I}_{e_6}\oplus\mathcal{I}_{e_7}$};
            \node(v2) at (3,1.5){$\mathcal{I}_{e_0}\oplus\mathcal{I}_{e_1}\oplus\mathcal{I}_{e_2}\oplus\mathcal{I}_{e_3}$};
            \node(v3) at (5.5,1.5){0};
            
            \draw[->]  (v0) edge node[above]{$\begin{pmatrix} 1 & 1 & 1 & 1 \\ \end{pmatrix}$} (v1);
            
            \draw[->]  (v1) edge node[above]{$\begin{pmatrix} -\partial_{40} & 0 & 0 & \partial_{70} \\\partial_{41} & -\partial_{51} & 0 & 0\\0 & \partial_{52} & -\partial_{62} & 0 \\ 0 & 0 & \partial_{63} & -\partial_{73} \end{pmatrix}$}(v2);
            \draw[->] (v2) edge (v3);

      \end{tikzpicture}
\]
Tracing this complex to $D(A^c_{Ent})$, we obtain:
\[
 \begin{tikzpicture}
            \node (v0) at (-7,1.5) {$P_{[1]}\oplus P_{[2]}\oplus P_{[1]}\oplus P_{[2]}$};
            \node(v1) at (3,1.5){$P_{[0]}\oplus P_{[1]}\oplus P_{[1]}\oplus P_{[1]}$};
            \node(v2) at (5.5,1.5){0};
            
            \draw[->]  (v0) edge node[above]{$\begin{pmatrix} -c & 0 & 0 & ba' \\1 & -c' & 0 & 0\\0 & b' & -1 & 0 \\ 0 & 0 & 1 & -c' \end{pmatrix}$}(v1);
            \draw[->] (v1) edge (v2);
      \end{tikzpicture}
\]
Applying ${-}\otimes_{A}{T}$ produces an explicit complex of line bundles quasi-isomorphic to the normalization of the cubic:  
\[
 \begin{tikzpicture}
            \node (v0) at (-7,1.5) {$\mathcal{O}(-1)\oplus\mathcal{O}(-2)\oplus\mathcal{O}(-1)\oplus\mathcal{O}(-2)$};
            \node(v1) at (3,1.5){$\mathcal{O}\oplus\mathcal{O}(-1)\oplus\mathcal{O}(-1)\oplus\mathcal{O}(-1)$};

            \draw[->]  (v0) edge node[above]{$\begin{pmatrix} -x_1 & 0 & 0 & x_0x_2 \\1 & -x_1 & 0 & 0\\0 & x_0 & -1 & 0 \\ 0 & 0 & 1 & -x_1 \end{pmatrix}$}(v1);
      \end{tikzpicture}
\]
One checks that the resolution is given by:
\[
 \begin{tikzpicture}
            \node (v0) at (-8,1.5) {$\mathcal{O}(-1)\oplus\mathcal{O}(-2)\oplus\mathcal{O}(-1)\oplus\mathcal{O}(-2)$};
            \node(v1) at (1,1.5){$\mathcal{O}\oplus\mathcal{O}(-1)\oplus\mathcal{O}(-1)\oplus\mathcal{O}(-1)$};
            \node(v2) at (6,1.5){$N$};
            
            \draw[->]  (v0) edge node[above]{$\begin{pmatrix} -x_1 & 0 & 0 & x_0x_2 \\1 & -x_1 & 0 & 0\\0 & x_0 & -1 & 0 \\ 0 & 0 & 1 & -x_1 \end{pmatrix}$}(v1);
            \draw[->] (v1) edge node[above]{$\begin{pmatrix} 1 & x_1 & t & t \end{pmatrix}$}(v2); 
      \end{tikzpicture}
\]
\end{example}

\end{document}